\documentclass[reqno]{article}

\usepackage{amssymb,amsmath,amsthm}
\usepackage{cite}
\usepackage[abbrev]{amsrefs}

\newtheorem{theo}{Theorem}[section]
\newtheorem{lemm}[theo]{Lemma}

\numberwithin{equation}{section}

\theoremstyle{definition}

\newtheorem{rema}[theo]{Remark}

\newcommand{\BC}{\mathbb{C}}

\newcommand{\BN}{\mathbb{N}}
\newcommand{\BR}{\mathbb{R}}

\newcommand{\bb}{\boldsymbol{b}}
\newcommand{\bD}{\boldsymbol{D}}

\newcommand{\bF}{\boldsymbol{F}}
\newcommand{\bI}{\boldsymbol{I}}
\newcommand{\bK}{\boldsymbol{K}}

\newcommand{\bmm}{\boldsymbol{m}}
\newcommand{\bN}{\boldsymbol{N}}

\newcommand{\bS}{\boldsymbol{S}}
\newcommand{\bT}{\boldsymbol{T}}
\newcommand{\bu}{\boldsymbol{u}}
\newcommand{\bU}{\boldsymbol{U}}
\newcommand{\bv}{\boldsymbol{v}}

\renewcommand{\L}{\mathrm{L}}
\renewcommand{\H}{\mathrm{H}}
\newcommand{\B}{\mathrm{B}}
\newcommand{\W}{\mathrm{W}}
\newcommand{\C}{\mathrm{C}}

\newcommand{\CF}{\mathcal{F}}

\newcommand{\pd}{\partial}
\newcommand{\wt}{\widetilde}
\newcommand{\wh}{\widehat}
\newcommand{\dv}{\mathrm{div}\,}
\newcommand{\supp}{\mathrm{supp}\,}
\newcommand{\dx}{\,\mathrm{d}x}

\begin{document}
\title{\Large Global existence of the Navier-Stokes-Korteweg equations with a non-decreasing pressure in $\mathrm{L}^p$-framework \footnotemark[0]}
\author{Keiichi Watanabe$^1$}
\date{}
\maketitle

\footnotetext[1]{keiichi-watanabe@akane.waseda.jp}

\vspace{-4mm}
\centerline{\textit{Department of Pure and Applied Mathematics, Waseda University}}

\abstract{\noindent We consider the isentropic Navier-Stokes-Korteweg equations with a non-decreasing pressure on the whole space $\BR^n$ $(n \ge 2)$, where the system describes the motion of compressible fluids such as liquid-vapor mixtures with phase transitions including a variable internal capillarity effect. We prove the existence of a unique global strong solution to the system in the $\mathrm{L}^p$-in-time and $\mathrm{L}^q$-in-space framework, especially in the maximal regularity class, by assuming $(p, q) \in (1, 2) \times (1, \infty)$ or $(p, q) \in \{2\} \times (1, 2]$. We show that the system is globally well-posed for small initial data belonging to $\mathrm{H}^{s + 1, q} (\BR^n) \times \mathrm{H}^{s, q} (\BR^n)^n$ provided $s \ge n/q$ if $q \le n$ and $s \ge 1$ if $q > n$. Our results allow the case when the derivative of the pressure is zero at a given constant state, that is, the critical states that the fluid changes a phase from vapor to liquid or from liquid to vapor. The arguments in this paper do not require any exact expression or a priori assumption on the pressure.}
	
%\keywords{Compressible fluids, Capillarity, Koteweg model, Diffuse interface model, Global well-posed.}

%---------------------------------------------------------------------------------------------------------------
%	1. introduction
%---------------------------------------------------------------------------------------------------------------

\section{Introduction}
\noindent
This paper investigates the motion of isentropic compressible fluids with capillary effects on the whole space $\BR^n$ $(n \ge 2)$:
\begin{align}
\label{2.1.1}
\left\{\begin{aligned}
\pd_t \varrho + \dv \bmm = 0, \\
\pd_t \bmm + \dv \left(\frac{\bmm \otimes \bmm}{\varrho}\right) - \dv \bT (\varrho, \bmm) + \nabla P(\varrho) = 0, \\
(\varrho, \bmm) \rvert_{t = 0} = (\varrho_0 (x), \bmm_0 (x)),
\end{aligned}\right.
\end{align}
where $\varrho = \varrho (x, t)$ and $\bmm = \bmm (x, t) = {}^\top\!(m_1 (x, t), \dots, m_n (x, t))$ stand the unknown density and momentum, respectively, at time $t > 0$ and the spatial
coordinate $x \in \BR^n$ $(n \ge 2)$; $\varrho_0 = \varrho_0 (x)$ and $\bmm_0 = \bmm_0 (x) = {}^\top\!(m_{0, 1} (x), \dots, m_{0, n} (x) )$ stand given initial data; $\bT (\bmm/\varrho)$ stands the  stress tensor defined by
\begin{align*}
\bT (\varrho, \bmm) & = \bS \left(\frac{\bmm}{\varrho}\right) + \bK (\varrho), \\
\bS \left(\frac{\bmm}{\varrho}\right) & = 2 \mu \bD \left(\frac{\bmm}{\varrho}\right) + (\nu - \mu) \left(\dv \left(\frac{\bmm}{\varrho}\right)\right) \bI, \\
\bK (\varrho) & = \frac{\kappa}{2} (\Delta \varrho^2 - \lvert \nabla \varrho \rvert^2) \bI - \kappa \nabla \varrho \otimes \nabla \varrho,
\end{align*}
where $\bD (\bmm/\varrho) = 2^{-1} (\nabla (\bmm/\varrho) + {}^\top\!(\nabla (\bmm/\varrho)))$ is the deformation tensor; $P(\varrho)$ stands the pressure that is assumed to be a smooth function with respect to $\varrho$ satisfying $P'(\rho_*) \ge 0$, where $\rho_*$ is a given positive constant; $\mu$ and $\nu$ stand given constants denoting the viscosity coefficients and $\kappa$ stands given constant denoting the capillary coefficient, which satisfy the conditions:
\begin{align}
\label{cond-1}
\mu > 0, \quad \nu > \frac{n - 2}{n} \mu, \quad \text{and} \quad \kappa > 0.
\end{align}
Here, the first and second conditions of \eqref{cond-1} ensure strong ellipticity for the Lam\'{e} operator. We emphasize that any exact expression or a priori assumption of the pressure $P$ is not imposed in the present paper. \par
The system \eqref{2.1.1} is said to be the \textit{Navier-Stokes-Korteweg equations}, which were originally introduced by Korteweg~\cite{K}. Notice that, when $\kappa = 0$, the system \eqref{2.1.1} deduces to the usual compressible Navier-Stokes equations. If $\kappa$ depends on the density $\varrho$, it is known that there are applications to shallow water models~\cite{BDL03,H10,JLW14} or quantum fluids models~\cite{BGV19,H17}. A feature of the Korteweg-type model is that the density gradient $\nabla \varrho$ appears in the stress tensor, which is a consequence of the second gradient theory due to van der Walls~\cite{R79}. The tensor $\bK$ containing $\nabla \varrho$ is called the \textit{Koteweg tensor} formulated by Dunn and Serrin~\cite{DS85} in order to analyze the structure of liquid-vapor mixtures with phase transitions. In particular, they proposed a concept of the \textit{interstitial working}, which is a mechanical quantity, and determine the coefficients of the Korteweg tensor such that the Clausius-Duhem inequality holds. Roughly speaking, for the Korteweg-type fluids, the entropy flux does not obey to the classical Fourier law but the ``Korteweg'' law---the entropy flux consists of the heat transfer and the interstitial working. For further details on the interstitial working, the reader may consult the paper by Dunn~\cite{D86}, Dunn and Serrin~\cite{DS85}, and references therein. Although the formulation due to Dunn and Serrin~\cite{DS85} assumed that $\kappa$ is a non-negative constant, Heida and M\'{a}lek~\cite{HM10} showed that it is allowed to consider $\kappa$ as a non-negative function of $\varrho$ with a slight modifications on the Korteweg tensor, where the term $\frac{1}{2} \varrho \kappa'(\varrho) \lvert \nabla \varrho \rvert^2 \bI$ is added to $\bK$. In the present paper, however, we suppose that the contribution of $\varrho$ to $\kappa$ is so small that we can see that $\kappa$ is a constant, which is the similar assumption as for the viscosity coefficients. Since the Cauchy problem for~\eqref{2.1.1} in the case when $\kappa \equiv 0$ is well-studied~\cite{CD10,CMZ10,CZ19,D00,DM17,FZZ18,FNP01,L98,MN80,M03}, this paper focuses on the case when $\kappa$ is a positive constant. \par
There are two different type of models to interpret liquid-vapor phase interfaces: \textit{sharp interface models} and \textit{diffuse interface models}. We here mention that diffuse interface models are also called phase field models. Sharp interface models treat the interfaces as zero thickness, which were introduced to investigate capillary effects occurring at the interface, especially, to consider the boundary conditions such as the Young-Laplace law in the early 19th century (cf. Finn~\cite{F86}). When we adopt sharp interface models on the fluid, we recognize the interface separates the fluid into liquid or vapor without spinodal phases taken into account, so that physical quantities, \textit{e.g.,} the density and momentum, are \textit{discontinuous} across the interfaces. To the best of our knowledge, sharp interface models have been known as one of the well-adopted models for moving boundary problems of two-phase flows~\cite{ACFGM,IH11,PS16,W18}, but there is an example that this model collapses. Indeed, when we simulate the moving contact line along a rigid surface, the thickness of interfaces cannot be negligible due to some numerical reasons~\cite{AMW98,LLGH15}. Compared with sharp interface models, in diffuse interface models, the liquid-vapor phase interfaces are construed as \textit{interfacial layers}, which are also called transition zone, with extremely small but non-zero thickness. Diffuse-interface models provide smoothness of physical quantities across the interfaces. More precisely, a order parameter varies sharply but smoothly, which enables us to overcome some difficulties arisen in sharp interface models~\cite{AMW98,LLGH15}. It is known that the Navier-Stokes-Korteweg equations \eqref{2.1.1} can be understood as a diffuse interface model due to the capillary terms, see, \textit{e.g.,} \cite{AMW98,CDMR05,LMM11,LLGH15}. We remark that, although the Navier-Stokes-Cahn-Hilliard equations have also been adopted to describe liquid-vapor mixtures (cf. Gurtin~\textit{et al.}~\cite{G96,GPV96} and Oden~\textit{et al.}~\cite{OHP10}), this system can be \textit{reduced} to the Navier-Stokes-Korteweg system by choosing suitable phase transformation rates and capillary stress tensors, see \cite{FK17,LT98}. We mention that this result is physically reasonable because the derivation of Cahn-Hilliard equations~\cite{CH58} are based on the Ginzburg-Landau phase transition theory~\cite{GL09}, which is essentially owing to van der Waals's idea~\cite{R79}. \par 
The diffuse interface model takes the density $\varrho$ as an order parameter to identify the phases. As far as explaining a background of phase transitions on the fluid, let $\varrho$ be a smooth function and bounded satisfying $0 < \varrho < b$ with some constant $b > 0$. Furthermore, let $T_* > 0$ be a constant standing a fixed temperature such that the fluid can be both vapor and liquid states. For this fixed temperature $T_*$, we define the Hemholtz free energy $W \in \C^2 ((0, b))$ by $W (\varrho) := \varrho w (\varrho)$, where $w (\,\cdot\,) \in \C^2 ((0, b))$ is some given function depending only on physical properties of the fluid. If the fluid allows phase transitions, the functional $W$ \textit{should} satisfy the following conditions:
\begin{enumerate}
\item The Helmholtz energy $W (\varrho)$ is a double-well potential. Namely, there exist constants $a_1, a_2 \in (0, b)$ such that $W'' (\varrho) > 0$ in $\varrho \in (0, a_1) \cup (a_2, b)$ and $W'' (\varrho) < 0$ in $\varrho \in (a_1, a_2)$.
\item For all $\varrho \in (0, b)$, we have $W (\varrho) \ge 0$ and $\lim_{\varrho \downarrow 0} W (\varrho) = \lim_{\varrho \uparrow b} W (\varrho) = + \infty$.
\end{enumerate}
Then, the phase of fluid with the density $\varrho$ can be classified into vapor, spinodal, and liquid phases if the value of the density $\varrho$ is contained in the interval $(0, a_1]$, $(a_1, a_2)$, and $[a_2, b)$, respectively. From the usual thermodynamical theory, it is known that the pressure $P$ can be written by
\begin{align*}
P (\varrho) = \varrho W' (\varrho) - W (\varrho) \quad \text{for $\varrho \in (0, b)$}.
\end{align*}
We notice that the pressure $P$ is still positive but its derivative may be negative, that is, the pressure is a \textit{non-monotone function} because $P' (\varrho) = \varrho W'' (\varrho)$ and the sign of $W'' (\varrho)$ may change. We remark that the most typical model describing liquid-vapor phase transitions is derived from the \textit{van der Waals law}: The pressure $P (\varrho)$ is given by
\begin{align}
\label{P}
P (\varrho) := R T_* \frac{b}{b - \varrho} - a \varrho^2,
\end{align}
where $R$ is the specific gas constant and $a, b$ are some positive constants independent of $\varrho$ and depending only on the physical quantity of the fluid. Defining
\begin{align}
\label{W}
W (\varrho) := \varrho \left(\int^\varrho_{\rho_*} \frac{P (\vartheta)}{\vartheta} \,\mathrm{d}\vartheta - \frac{P (\rho_*)}{\rho_*} \right),
\end{align}
we observe that the functional \eqref{P} and \eqref{W} satisfy the required properties such that phase transitions occur. Notice that $W$ is convex when the pressure $P$ is non-decreasing, which includes the standard isentropic pressure laws of $\gamma$-type. Summing up, our assumption $P'(\rho_*) \ge 0$ denotes that the fluid is not spinodal states but liquid or vapor states, which allows the \textit{critical} case $P'(\rho_*) = 0$. For more information on the background of phase transitions, we refer to \cite{AMW98,CDMR05,O02,R05} and references therein. The author remark that it is expected to show \textit{unstable} global-in-time solutions to the system \eqref{2.1.1} in the case when $P'(\rho_*) < 0$ because the fluid is spinodal phases---both vapor and liquid coexist precariously in the thermodynamical sense. Indeed, when $P'(\rho_*) < 0$, there exists an eigenvalue of the corresponding linearised equations to~\eqref{2.1.1} such that its real part is positive for some frequencies. For the local existence result on the case $P'(\rho_*) < 0$, we refer to Charve~\cite{C14}. \par
Concerning the Cauchy problem \eqref{2.1.1} on the whole space $\BR^n$ $(n \ge 2)$, the existence of global-in-time solutions has been shown by several authors, see \cite{CDX18,C16,CH11,CH13,CK19,DD01,HL94,HL96,HL96b,HPZ18,KT19,TWX12}. We emphasize that \textit{all} these results were obtained in the $\L^2$ framework because their proof strongly rely on the energy method or the Plancherel theorem, which is only applicable within the $\L^2$ framework in general. Hence, another new technique is required to show a unique global solvability to the system \eqref{2.1.1} in the $\L^p$-in-time and $\L^q$-in-space framework. Recently, Murata and Shibata~\cite{MS19} overcame this difficulty by using the maximal regularity theory and proved that the system~\eqref{2.1.1} admits a global-in-time solution $(\varrho, \bv)$ satisfying
\begin{align*}
\varrho - \rho_* & \in \W^{1, p} ((0, \infty); \W^{1, q} (\BR^n)) \cap \L^p ((0, \infty); \W^{3, q} (\BR^n)), \\
\bv & \in\W^{1, p} ((0, \infty); \L^q (\BR^n)^n) \cap \L^p ((0, \infty); \W^{2, q} (\BR^n)^n)
\end{align*}
with $\bmm = \varrho \bv$ and some given constant $\rho_*$ by assuming $2 < p < \infty$ and $n < q < \infty$ with $2 / p + n / q < 1$, whose function spaces are said to be the \textit{maximal regularity class}. However, this result does not cover any $\L^2$ framework result because the assumptions on $(p, q)$ guarantee some embedding properties in order to bound the nonlinear terms in a suitable norm, so that it does not admit to take $(p, q) = (2, 2)$. We remark that Charve~\textit{et al.}~\cite{CDX18} establish the global existence theorem in the $\L^p$ framework, but their result imposed the $\L^2$ integrability for initial data with respect to space variables on the low frequencies. \par
The aim of this paper is to show a unique global solvability of the system \eqref{2.1.1} in the $\L^p$-in-time and $\L^q$-in-space framework, especially in the maximal regularity class including the $\L^2$ framework results. In this paper, especially, we seek a solution~$(\varrho, \bmm)$ to the system~\eqref{2.1.1} around a constant state $(\rho_*, 0)$. To this end, we suppose that the initial density can be written in the from of $\varrho_0 (x) = \rho_* + \pi_0 (x)$ with some given function $\pi_0 (x)$. To simplify the notation, we normalize $\rho_*$ to $1$ by an appropriate rescaling. The following is the main theorem of this paper.
\begin{theo}
\label{Th-2.1.1}
Let $\mu$, $\nu$, $\kappa$ satisfy \eqref{cond-1}. Suppose that $P$ is a given smooth function with respect to $\varrho$ such that $P'(1) \ge 0$. If $P'(1) = 0$, we additionally assume that $(\mu + \nu)^2 \ge 4 \kappa$. Let $p$ and $q$ satisfy	
\begin{align}
\label{cond-212}
(p, q) = (1, 2) \times (1, \infty) \quad \text{or} \quad (p, q) = \{2\} \times (1, 2]
\end{align}
and $s$ satisfy
\begin{align}
\label{cond-s}
\begin{cases}
s > n/q & \text{if $q \le n$}, \\
s \ge 1 & \text{if $q > n$}.
\end{cases}
\end{align}
There exists a positive constant $C$ such that if $\lVert \varrho_0 - 1 \rVert_{\H^{s + 1, q} (\BR^n)} + \lVert \bmm_0 \rVert_{\H^{s, q} (\BR^n)} < C$ holds, then the system \eqref{2.1.1} admits a unique global strong solution $(\varrho, \bmm)$ satisfying
\begin{align*}
\varrho - 1 & \in \W^{1, p} ((0, \infty); \H^{s, q} (\BR^n)) \cap \L^p ((0, \infty); \H^{s + 2, q} (\BR^n)), \\
\bmm & \in \W^{1, p} ((0, \infty); \H^{s - 1, q} (\BR^n)^n) \cap \L^p ((0, \infty); \H^{s + 1, q} (\BR^n)^n).
\end{align*}	
\end{theo}

\begin{rema}
We shall give some comments on Theorem \ref{Th-2.1.1}.
\begin{enumerate}\renewcommand{\labelenumi}{(\arabic{enumi})}
\item The assumption $P'(1) \ge 0$ is a sufficient condition such that eigenvalues of the corresponding linearised system to \eqref{2.1.1} have positive real parts for \textit{all} frequencies. As we mentioned above, this assumption is reasonable from the view point of the thermodynamics. However, if $P'(1) = 0$, we suppose $(\mu + \nu)^2 \ge 4 \kappa$. This additional assumption will guarantee the Fourier multiplier of the solution operator satisfies the Mikhin condition.
\item It suffices to suppose $P$ is a $\C^{[s] + 1}$-function, see Lemma \ref{lem-2.5.2} below.
\item According to the trace method of real interpolation, for $1 < p < \infty$ it hold
\begin{align}
\label{emb-MR}
\begin{aligned}
& \W^{1, p} ((0, \infty); \W^{1, q} (\BR^n)) \cap \L^p ((0, \infty); \W^{3, q} (\BR^n)) \\
& \quad \hookrightarrow \mathrm{BUC} ([0, \infty); \B^{s + 2 - 2/p}_{q, p} (\BR^n)), \\
& \W^{1, p} ((0, \infty); \H^{s - 1, q} (\BR^n)^n) \cap \L^p ((0, \infty); \H^{s + 1, q} (\BR^n)^n) \\
& \quad \hookrightarrow \mathrm{BUC} ([0, \infty); \B^{s + 1 - 2/p}_{q, p} (\BR^n)),
\end{aligned}
\end{align}
see Amann \cite[Theorem III.~4.10.2]{A95} (cf. Tanabe~\cite[pp.~10]{T97}). Hence, Theorem \ref{Th-2.1.1} yields that the system~\eqref{2.1.1} is globally well-posed for small initial data belonging to $\H^{s + 1, q} (\BR^n) \times \H^{s, q} (\BR^n)^n$ assuming that  \eqref{cond-212} and \eqref{cond-s} because we have $\H^{s, q} (\BR^n) \hookrightarrow \B^{s + 1 - 2/p}_{q, p} (\BR^n)$ and $\H^{s + 1, q} (\BR^n) \hookrightarrow \B^{s + 2 - 2/p}_{q, p} (\BR^n)$, see, \textit{e.g.}, Pr{\"u}ss and Simonett~\cite[Chapter~4]{PS16} for further general results.
\end{enumerate}
\end{rema}
The rest of this paper is constructed as follows: The next section introduces notation and lemmas which we use throughout this paper. In Section \ref{subsec-2.3}, we focus on the linearized system and give the results for the corresponding eigenvalues. Section \ref{subsec-2.4} is devoted to show the linear estimates, which play an important role for our discussion on proving Theorem \ref{Th-2.1.1} in Section \ref{subsec-2.5}. 

%---------------------------------------------------------------------------------------------------------------
%	2. Preliminaries
%---------------------------------------------------------------------------------------------------------------

\section{Preliminaries}
\subsection{Notation}
In this subsection, we introduce symbols and function spaces we use throughout this paper. The symbols $\BN$, $\BR$ and $\BC$ is the set of all natural, real, and complex numbers, respectively. We define $\BN_0 := \BN \cup \{0\}$ and $\BR_+ := (0, \infty)$. \par
For a Banach space $X$, let $X^n$ be the $n$-product space of $X$ and its norm be denoted by $\lVert \, \cdot \, \rVert_X$ for short. The identity mapping on $X$ is denoted by $\bI$ when no confusion is possible. \par
The Fourier transform on $\BR^n$ is defined by
\begin{align*}
\CF[f] (\xi) = \wh f (\xi) := \int_{\BR^n} f (x) e^{- i x \cdot \xi} \dx
\end{align*}
and the inverse Fourier transform of $f$ is defined by
\begin{align*}
\CF^{- 1} [f] (x) := \frac{1}{(2 \pi)^n} \int_{\BR^n} f (x) e^{i x \cdot \xi} \,\mathrm{d}\xi,
\end{align*}
where $x \cdot \xi$ is the dot product defined by $x \cdot \xi := x_1 \xi_1 + \cdots + x_n \xi_n$ for $x = (x_1, \dots, x_n)$ and $\xi = (\xi_1, \dots, \xi_n)$. \par
The symbols $\L^p (\BR^n)$, $\H^{s, p} (\BR^n)$, and $\B^s_{q, p} (\BR^n)$ stand the usual Lebesgue, Bessel potential, and Besov spaces, respectively. For a Banach space $X$, we denote $\L^p (I; X)$ by the standard Bochner-Lebesgue spaces, where $I \subset \BR_+$ is an open time interval and $T > 0$. Similarly, functions spaces $\W^{1, p} (I; X)$ describe the set of all functions $u$ such that $u, \pd_t u \in \L^p (I; X)$.

\subsection{Auxiliary lemmas}
In this subsection, we give auxiliary lemmas we use in below. It is well-known that the negative of the Laplace operator admits the \textit{maximal regularity}, see, \textit{e.g.}, Hieber and Pr{\"u}ss~\cite{HP97} and Pr{\"u}ss and Simonett~\cite[Chapter 4]{PS16}.
\begin{lemm}
\label{lem-MR}
Let $1 < p, q < \infty$ and $f \colon \BR_+ \to \L^p (\BR_+ ; \L^q (\BR^n))$ be a given function. Consider the inhomogeneous heat equation:
\begin{align}
\label{eq-heat}
\pd_t u - \Delta u = f, \qquad u (x, 0) = u_0 (x)
\end{align}
for $u_0 \in \B^{2 (1 - 1/p)}_{q, p} (\BR^n)$. Then the problem \eqref{eq-heat} admits a unique strong solution $u$ in the maximal regularity class,
\begin{align*}
u \in \W^{1, p} (\BR_+ ; \L^q (\BR^n)) \cap \L^p (\BR_+; \W^{2, q} (\BR^n)).
\end{align*}
Especially, we have the estimate
\begin{align*}
& \lVert u \rVert_{\L^p (\BR_+ ; \L^q (\BR^n))} + \lVert \pd_t u \rVert_{\L^p (\BR_+ ; \L^q (\BR^n))} + \lVert (- \Delta) u \rVert_{\L^p (\BR_+ ; \L^q (\BR^n))} \\
& \quad \le C \Big(\lVert u_0 \rVert_{\B^{2 (1 - 1/p)}_{q, p}(\BR^d)} + \lVert f \rVert_{\L^p (\BR_+ ; \L^q (\BR^n))}\Big)
\end{align*}
with some positive constant $C$.
\end{lemm}

We shall recall the \textit{Bernstein-type inequalities} which are convenient when we bound spectrally localized functions. The proof can be seen in Bahouri \textit{et al.}~\cite[Lemma 2.1]{BCD}.
\begin{lemm}
\label{lem-2.2}	
For $0 < r < R < \infty$. There exists a positive constant $C$ such that for any $k \in \BN_0$, for any $1 \le p \le q \le \infty$ and $f \in \L^p (\BR^n)$, and for some $\lambda > 0$, we have
\begin{align*}
\supp \wh f (\xi) \subset \{\xi \in \BR^n \colon \lvert \xi \rvert \le \lambda r \} & \Longrightarrow \sup_{\lvert a \rvert = k} \lVert \pd^a_x f \lVert_{\L^q (\BR^n)} \le C^k \lambda^{k + 3 (\frac{1}{p} - \frac{1}{q})} \lVert f \rVert_{\L^p (\BR^n)}, \\
\supp \wh f (\xi) \subset \{\xi \in \BR^n \colon \lambda r \le \lvert \xi \rvert \le \lambda R \} & \Longrightarrow C^{- k} \lambda^k \lVert f \rVert_{\L^p (\BR^n)} \le \sup_{\lvert a \rvert = k} \lVert \pd^a_x f \rVert_{\L^p (\BR^n)} \le C^k \lambda^k \lVert f \rVert_{\L^p (\BR^n)},
\end{align*}
where $s \in \BN_0^n$ is a multi-index. Here, the constant $C$ is independent of $\lambda$ and $f$.
\end{lemm}

%---------------------------------------------------------------------------------------------------------------
%	3. Linearized system
%---------------------------------------------------------------------------------------------------------------
\section{Linearized system}
\label{subsec-2.3}
\noindent
In this section, we introduce the linearised system associated with \eqref{2.1.1}. Noting $\varrho = 1 + \pi$, we see that the system~\eqref{2.1.1} can be reformulated in the form of
\begin{align}
\label{eq-221}
\left\{
\begin{aligned}
\pd_t \pi + \dv \bmm & = 0, \\
\pd_t \bmm - \big(\mu \Delta \bmm + \nu \nabla \dv \bmm - \gamma \nabla \pi + \kappa \nabla \Delta \pi \big) & = \bF (\pi, \bmm), \\
(\pi, \bmm) \rvert_{t = 0} & = (\pi_0, \bmm_0),
\end{aligned}\right.
\end{align}
with
\begin{align*}
\bF (\pi, \bmm) & = - \dv \bigg(\frac{1}{1 + \pi} \bmm \otimes \bmm \bigg) - \mu \Delta \bigg(\frac{\pi}{1 + \pi} \bmm \bigg) - \nu \nabla \dv \bigg(\frac{\pi}{1 + \pi} \bmm \bigg) \\
& \quad - (G (\pi) + P''(1) \pi) \nabla \pi  - \kappa \dv \left\{\left((\pi \Delta \pi) - \frac{\lvert \nabla \pi \rvert^2}{2}\right) \bI - \nabla \pi \otimes \nabla \pi \right\},
\end{align*}
where $\gamma = P'(1)$ and
\begin{align*}
G (\theta) = P' (1 + \theta) - \gamma - P''(1) \theta \qquad \text{for $\theta \in \BR$}.
\end{align*}
We remark that $G (0) = G'(0) = 0$ and $G$ is independent of $t$. We then obtain the linearized system of~\eqref{eq-221}:
\begin{align}
\label{eq-222}
\left\{
\begin{aligned}
\pd_t
\begin{pmatrix}
\phi \\ \bu
\end{pmatrix}
- \begin{pmatrix}
0 & - \dv \\
- \gamma \nabla + \kappa \nabla \Delta & \mu \Delta \bI + \nu \nabla \dv
\end{pmatrix}
\begin{pmatrix}
\phi \\ \bu
\end{pmatrix}
& =
0, \\
(\phi, \bu) \rvert_{t = 0} & = (\phi_0, \bu_0),
\end{aligned}\right.
\end{align}
where $\bu (x, t) = {}^\top\!(u_1 (x, t), \dots, u_n (x, t))$ and $\bu_0 (x) = {}^\top\!(u_{0, 1} (x), \dots, u_{0, n} (x))$. We mention that it is fail to adopt a general theory for symmetric hyperbolic-parabolic systems established by Shizuta and Kawashima~\cite{SK85} due to the capillary term. Applying the Fourier transform with respect to $x$ implies
\begin{align}
\label{eq-223}
\left\{
\begin{aligned}
\pd_t
\begin{pmatrix}
\wh \phi \\ \wh \bu
\end{pmatrix}
- \begin{pmatrix}
0 & - \mathrm{i} {}^\top\! \xi \\
- (\gamma + \kappa \lvert \xi \rvert^2) \mathrm{i} \xi & - \mu \lvert \xi \rvert^2 \bI - \nu \xi {}^\top\! \xi
\end{pmatrix}
\begin{pmatrix}
\wh \phi \\ \wh \bu
\end{pmatrix}
& =
0, \\
(\wh \phi, \wh \bu) \rvert_{t = 0} & = (\wh \phi_0, \wh \bu_0).
\end{aligned}\right.
\end{align}
Let $\det (\lambda)$ be the polynomial of $\lambda \in \BC$ defined by
\begin{align}
\label{224}
\det (\lambda) := \lambda^2 + (\mu + \nu) \lvert \xi \rvert^2 \lambda + (\gamma + \kappa \lvert \xi \rvert^2) \lvert \xi \rvert^2.
\end{align}
The equation $\det (\lambda )= 0$ has two roots:
\begin{equation}
\begin{split}
\label{225}
\lambda_\pm (\xi) & = - \frac{\mu + \nu}{2} \lvert \xi \rvert^2 \pm \frac{\mu + \nu}{2} \sqrt{\lvert \xi \rvert^4 - \frac{4 (\gamma + \kappa \lvert \xi \rvert^2) \lvert \xi \rvert^2}{(\mu + \nu)^2}} \\
& = - \frac{\mu + \nu}{2} \lvert \xi \rvert^2 \pm \frac{\mu + \nu}{2} \sqrt{\left(1 - \frac{4 \kappa}{(\mu + \nu)^2}\right) \lvert \xi \rvert^4 - \frac{4 \gamma \lvert \xi \vert^2}{(\mu + \nu)^2}} \\
& =: - A \lvert \xi \rvert^2 \pm A \sqrt{\left(1 - \frac{\kappa}{A^2}\right) \lvert \xi \rvert^4 - \frac{\gamma}{A^2} \lvert \xi \rvert^2},
\end{split}
\end{equation}
where we have set $A := - (\mu + \nu)/2$ for short. The behaviors of $\lambda_\pm (\xi)$ are classified into the following cases:
\begin{alignat*}4
& \textbf{Case 1:} &\enskip  \text{$A^2 > \kappa$ and $\gamma > 0$}; & \qquad \textbf{Case 2:} &\enskip  \text{$A^2 < \kappa$ and $\gamma > 0$}; \\
& \textbf{Case 3:} &\enskip  \text{$A^2 = \kappa$ and $\gamma > 0$}; & \qquad \textbf{Case 4:} &\enskip  \text{$A^2 > \kappa$ and $\gamma = 0$}; \\
& \textbf{Case 5:} &\enskip  \text{$A^2 < \kappa$ and $\gamma = 0$}; & \qquad \textbf{Case 6:} &\enskip  \text{$A^2 = \kappa$ and $\gamma = 0$}.
\end{alignat*}
In this paper, we exclude Case 5.
\par
Let us focus on Case 1. We see that $\lambda_+ (\xi) = \lambda_- (\xi)$ holds if $\xi$ satisfies
\begin{align*}
\lvert \xi \rvert = \sqrt{\frac{\gamma}{A^2 - \kappa}} =: B.
\end{align*}
Noting \eqref{225}, we observe that
\begin{align}
\label{226}
\lambda_\pm (\xi) =
\left\{
\begin{aligned}
& - A \lvert \xi \rvert^2 \pm \mathrm{i} \sqrt \gamma \lvert \xi \rvert \sqrt{1 - \frac{\lvert \xi \rvert^2}{B^2}} & \enskip & \text{for $0 < \lvert \xi \rvert < B$}, \\
& - A \lvert \xi \rvert^2 \pm \frac{1}{\sqrt \gamma B} \lvert \xi \rvert^2 \sqrt{1 - \frac{B^2}{\lvert \xi \rvert^2}} & \enskip & \text{for $\lvert \xi \rvert > B$}. \\
\end{aligned}
\right.
\end{align}
Especially, in this case, the eigenvalues $\lambda_\pm (\xi)$ have the following asymptotic behaviors:
\begin{align*}
\lambda_\pm (\xi) =
\left\{
\begin{aligned}
& \pm \mathrm{i} \sqrt \gamma \lvert \xi \rvert - A \lvert \xi \rvert^2 \mp \mathrm{i} \frac{\sqrt \gamma}{B^2} \lvert \xi \rvert^3 + \mathrm{i} O(\lvert \xi \rvert^5) & \quad & \text{as $\lvert \xi \rvert \to 0$}, \\
& - A \left(1 \mp \sqrt{1 - \frac{\kappa}{A^2}}\right) \lvert \xi \rvert^2 \mp \frac{\sqrt \gamma}{2 B} + O\bigg(\frac{1}{\lvert \xi \rvert^2}\bigg) & \quad & \text{as $\lvert \xi \rvert \to \infty$}
\end{aligned}
\right.
\end{align*}
and the solution to \eqref{eq-223} is represented by
\begin{equation}
\label{227}
\begin{split}
\wh \phi (\xi, t) & = \left(\cfrac{\lambda_+ (\xi) e^{\lambda_- (\xi) t} - \lambda_- (\xi) e^{\lambda_+ (\xi) t}}{\lambda_+ (\xi) - \lambda_- (\xi)}\right) \wh \phi_0 (\xi) \\
& \quad - \left(\cfrac{e^{\lambda_+ (\xi) t} - e^{\lambda_- (\xi) t}}{\lambda_+ (\xi) - \lambda_- (\xi)}\right) (\mathrm{i} \xi \cdot \wh \bu_0 (\xi)), \\
\wh \bu (\xi, t) & = - \left(\cfrac{e^{\lambda_+ (\xi) t} - e^{\lambda_- (\xi) t}}{\lambda_+ (\xi) - \lambda_- (\xi)} \right) (\gamma + \kappa \lvert \xi \rvert^2) \mathrm{i} \xi \wh \phi_0 (\xi) + e^{- \mu \lvert \xi \rvert^2 t} \wh \bu_0 (\xi) \\
& \quad + \left(\cfrac{\lambda_+ (\xi) e^{\lambda_+ (\xi) t} - \lambda_- (\xi) e^{\lambda_- (\xi) t}}{\lambda_+ (\xi) - \lambda_- (\xi)} - e^{- \mu \lvert \xi \rvert^2 t} \right) \cfrac{\xi (\xi \cdot \wh \bu_0 (\xi))}{\lvert \xi \rvert^2}
\end{split}
\end{equation}
for $\xi \in \BR^n \setminus \{0\}$. In our analysis, it is convenient to rewrite as
\begin{equation}
\label{228}
\begin{split}
\wh \phi (\xi, t) & = \sum_{\ell = \pm} e^{\lambda_\ell (\xi) t} G_{1, 1}^{(\ell)} (\xi) \wh \phi_0 (\xi) + \sum_{k = 1}^n \sum_{\ell = \pm} e^{\lambda_\ell (\xi) t} G_{1, k + 1}^{(\ell)} (\xi) \wh u_{0, j} (\xi), \\
\wh u_j (\xi, t) & = \sum_{\ell = \pm} e^{\lambda_\ell (\xi) t} G_{j + 1, 1}^{(\ell)} (\xi) \wh \phi_0 (\xi) + e^{- \mu \lvert \xi \rvert^2 t} \wh u_{0, j} (\xi) \\
& \quad + \sum_{k = 1}^n \bigg(\sum_{\ell = \pm} e^{\lambda_\ell (\xi) t} G_{j + 1, k + 1}^{(\ell)} (\xi) + e^{- \mu \lvert \xi \rvert^2 t} \frac{\xi_j \xi_k}{\lvert \xi \rvert^2} \bigg) \wh u_{0, k} (\xi)
\end{split}
\end{equation}
for $j = 1, \dots, n$, where $G_{JK}^{(\ell)} (\xi)$, $J, K = 1, \dots, n + 1$, are defined by
\begin{align}
\label{229}
\begin{aligned}
G^{(\pm)}_{1, 1} & = \mp \frac{\lambda_\mp (\xi)}{\lambda_+ (\xi) - \lambda_- (\xi)}, & \quad G^{(\pm)}_{1, k + 1} & = \mp \frac{\mathrm{i} \xi_k}{\lambda_\pm (\xi) - \lambda_- (\xi)}, \\
G^{(\pm)}_{j + 1, 1} & = \mp \frac{\lambda_\pm (\xi) (\gamma + \kappa \lvert \xi \rvert^2) \mathrm{i} \xi_j}{\lambda_+ (\xi) - \lambda_- (\xi)}, & \quad 
G^{(\pm)}_{j + 1, k + 1} & = \pm \frac{\lambda_\pm (\xi) \xi_j \xi_k}{\lvert \xi \rvert^2 (\lambda_+ (\xi) - \lambda_- (\xi))}.
\end{aligned}
\end{align}
We notice that, for Case 1, these formulas hold if $\lvert \xi \rvert \neq B$. On the other hand, if $B/2 < \lvert \xi \rvert < 2B$, we write the solution to \eqref{eq-223} in the form of
\begin{equation}
\label{2210}
\begin{split}
\wh \phi (\xi, t)
& = \left(\frac{1}{2 \pi \mathrm{i}}\oint_\Gamma \frac{(z + 2 A \lvert \xi \rvert^2) e^{z t}}{\det (z)} \,\mathrm{d}z \right) \wh \phi_0 (\xi) + \left(\frac{1}{2 \pi} \oint_\Gamma \frac{e^{z t}}{\det (z)} \,\mathrm{d}z \right) \xi \cdot \wh \bu_0 (\xi), \\
\wh \bu (\xi, t)
& = - \left(\frac{1}{2 \pi} \oint_\Gamma \frac{e^{z t}}{\det (z)} \,\mathrm{d}z \right) (\gamma + \kappa \lvert \xi \rvert^2)  \wh \phi_0 (\xi) + e^{- \mu \lvert \xi \rvert^2 t} \wh \bu_0 (\xi) \\
& \quad + \left(\frac{1}{2 \pi \mathrm{i}} \oint_\Gamma \frac{z e^{z t}}{\det (z)} \,\mathrm{d}z - e^{- \mu \lvert \xi \rvert^2 t} \right) \frac{\xi (\xi \cdot \wh \bu_0)}{\lvert \xi \rvert^2},
\end{split}
\end{equation}
where $\Gamma \subset \{z \in \BC \colon \mathrm{Re}\, z \le - c_0 \}$ denotes a closed pass including the zero points of $\det (z)$ with some positive constant $c_0$ independent of $t$  and $\xi$ such that
\begin{align}
\label{2211}
\max_{B / 2 \le \lvert \xi \rvert \le 2 B} \mathrm{Re}\, \lambda_\pm (\xi) \le - 2 c_0.
\end{align}
Here, the formulas \eqref{227}--\eqref{2210} can be derived in the same manner as in Hoff and Zumbrun~\cite[Section~3]{HZ95} and Kobayashi and Shibata~\cite[Section 2]{KS02}, see next section. We, in addition, remark that we have used Cauchy's integral expression to derive \eqref{2210}. \par
We next treat the other cases. For Case 2, we have
\begin{align*}
\lambda_\pm (\xi) = - A \lvert \xi \rvert^2 \pm \mathrm{i} A \lvert \xi \rvert \sqrt{\left(\frac{\kappa}{A^2} - 1 \right) \lvert \xi \rvert^2 + \frac{\gamma}{A}},
\end{align*}
which have the asymptotic behaviors:
\begin{align*}
\lambda_\pm (\xi) =
\left\{
\begin{aligned}
& \pm \mathrm{i} \gamma \lvert \xi \rvert - A \lvert \xi \rvert^2 \pm \mathrm{i} \frac{\kappa - A^2}{2 A} \lvert \xi \rvert^3 + \mathrm{i} O(\lvert \xi \rvert^5) & & \text{as $\lvert \xi \rvert \to 0$}, \\
& - A \lvert \xi \rvert^2 \pm \mathrm{i} \sqrt{\kappa - A^2} \lvert \xi \rvert^2 \pm \mathrm{i} \frac{A \gamma}{2 \sqrt{\kappa - A^2}} + \mathrm{i} O\bigg(\frac{1}{\lvert \xi \rvert^2}\bigg)& & \text{as $\lvert \xi \rvert \to \infty$}.
\end{aligned}
\right.
\end{align*}
For the rest cases, the eigenvalues can be written by
\begin{align*}
\lambda_\pm (\xi) =
\left\{
\begin{aligned}
& \pm \mathrm{i} \frac{\sqrt{\gamma}}{A} \lvert \xi \rvert - A \lvert \xi \rvert^2 & \quad & \text{for Case 3}, \\
& - A \left(1 \mp \sqrt{1 - \frac{\kappa}{A^2}} \right) \lvert \xi \rvert^2 & \quad & \text{for Case 4}, \\
& - A \lvert \xi \rvert^2 \pm \mathrm{i} \left(\sqrt{\frac{\kappa}{A^2} - 1}\right) \lvert \xi \rvert^2 & \quad & \text{for Case 5}, \\
& - A \lvert \xi \rvert^2 & \quad & \text{for Case 6},
\end{aligned}
\right.
\end{align*}
where we have a real double root for Case 6. For Case 2--4, the solution to \eqref{eq-223} is expressed in the form of \eqref{227}--\eqref{2210}. As for Case 6, the solution to \eqref{eq-223} is denoted by
\begin{equation}
\begin{split}
\label{2212}
\wh \phi (\xi, t) & = (1 + \mu \lvert \xi \rvert^2 t) e^{- \mu \lvert \xi \rvert^2} \wh \phi_0 - t e^{- \mu \lvert \xi \rvert^2 t} (i \xi \cdot \wh \bu_0), \\
\wh \bu (\xi, t) & = A^2 \lvert \xi \rvert^2 t e^{- A \lvert \xi \rvert^2 t} \mathrm{i} \xi \wh \phi_0  + e^{- \mu \lvert \xi \rvert^2 t} \wh \bu_0 \\
& \quad + \bigg((A \lvert \xi \rvert^2 t - 1) e^{- A \lvert \xi \rvert^2 t}  - e^{- \mu \lvert \xi \rvert^2 t} \bigg) \frac{\xi (\xi \cdot \wh \bu_0)}{\lvert \xi \rvert^2},
\end{split}
\end{equation}
see the appendix for derivations. According to the above analysis, we emphasize that we may expect a parabolic smoothing for the solution to \eqref{2.1.1} for \textit{all} frequencies, which enables us to prove the global solvability to \eqref{2.1.1} via the Banach fixed point iteration.

%\begin{rema}
%	If $\gamma < 0$, there exists an eigenvalue whose real part is positive. For example, if we consider the case when $\gamma < 0$ and $A^2 > \kappa$, the asymptotic behavior of $\lambda_\pm (\xi)$ is given by
%	\begin{align*}
%	\lambda_\pm (\xi) = \pm \frac{\lvert \gamma \rvert}{A} \lvert \xi \rvert \mp \frac{\kappa - A^2}{2 A} \lvert \xi \rvert^3 + O(\lvert \xi \rvert^5) \quad \text{as $\lvert \xi \rvert \to 0$}.
%	\end{align*}
%	Thus, there exists an \textit{unstable} solution to \eqref{eq-222}.
%\end{rema}

\section{Estimates for the solution operator}
\label{subsec-2.4}
\noindent
We define the solution operator $T(t) = {}^\top\! (T_\phi (t), T_{\boldsymbol{u}} (t))$ such that
\begin{align*}
T(t) \bU_0 (x) = \CF^{- 1} \left[\wh \bU (\xi, t) \right] (x) = \bU (x, t)
\end{align*}
for every $t > 0$, where $\bU := {}^\top\!(\pi, \bu)$ and $\bU_0 := {}^\top\!(\phi_0, \bu_0)$ are a solution and an initial data of the system~\eqref{eq-222}, respectively.  In this section, we prove the following lemmas.
\begin{lemm}
\label{lem-2.4.1}
Let $1 \le q \le r \le \infty$. Let $\gamma \ge 0$ and $A^2 \ge \kappa$ if $\gamma = 0$. There exists positive constants $c$ and $C$ independent of $t$ and $\bU_0$ such that
\begin{align*}
\left\lVert \pd_x^a T (t) \bU_0 \right\rVert_{\L^r (\BR^n)} \le C \left(t^{- \frac{n}{2} (\frac{1}{q} - \frac{1}{r})} + e^{- c t} + e^{- c t} t^{- \frac{n}{2} (\frac{1}{q} - \frac{1}{r}) - \frac{\lvert a \rvert}{2}} \right) \left\lVert \bU_0 \right\rVert_{\L^q (\BR^n)}
\end{align*}
holds for any $a \in \BN_0^n$, $t > 0$, and $\bU_0 \in \L^q (\BR^n)^{n + 1}$. 	
\end{lemm}
\begin{lemm}
\label{lem-2.4.2}
Let $(p, q)$ satisfy \eqref{cond-212}. Assume that $\gamma \ge 0$ and $A^2 \ge \kappa$ if $\gamma = 0$. Then, for any $\bU_0 \in \H^{1, q} (\BR^n)$, we have the estimates
\begin{align*}
\lVert T(t) \bU_0 \rVert_{\W^{1, p} (\BR_+; \L^q (\BR^n))} & \le C \lVert \bU_0 \rVert_{\H^{1, q} (\BR^n)}, \\
\lVert (- \Delta) T(t) \bU_0 \rVert_{\L^p (\BR_+; \L^q (\BR^n))} & \le C \left\lVert \bU_0 \right\rVert_{\H^{1, q} (\BR^n)},
\end{align*}
where $C$ is some positive constant independent of $t$.		
\end{lemm}
\begin{lemm}
\label{lem-2.4.3}
Let $1 < p, q < \infty$ satisfy \eqref{cond-212} and let $\gamma \ge 0$. If $\gamma = 0$, we additionally suppose $A^2 \ge \kappa$. There exists a positive constant $C$ independent of $t$ such that the estimates
\begin{align*}
\left\lVert \int_0^t T (t - \tau) f (\tau) \,\mathrm{d}\tau \right\rVert_{\W^{1, p} (\BR_+; \L^q (\BR^n))} & \le C \lVert f \rVert_{\L^p (\BR_+; \L^q (\BR^n))}, \\
\left\lVert \int_0^t (- \Delta) T (t - \tau) f (\tau) \,\mathrm{d}\tau \right\rVert_{\L^p (\BR_+; \L^q (\BR^n))} & \le C \lVert f \rVert_{\L^p (\BR_+; \L^q (\BR^n))}
\end{align*}	
are valid for any $f \in \L^p (\BR_+; \L^q (\BR^n))$.			
\end{lemm}

We first consider Case 1. To establish the estimates for $T (t)$, we decompose the operator $T(t)$ as follows: Let $\varphi_1, \varphi_\infty, \varphi_M \in \C^\infty (\BR^n)$ be cut-off functions defined by
\begin{align*}
\varphi_1 (\xi) & = \begin{cases}
1 & \quad \text{for} \quad \lvert \xi \rvert \le B/2, \\
0 & \quad \text{for} \quad \lvert \xi \rvert \ge B/\sqrt 2,
\end{cases}
\\
\varphi_\infty (\xi) & = \begin{cases}
1 & \quad \text{for} \quad \lvert \xi \rvert \ge 2 B, \\
0 & \quad \text{for} \quad \lvert \xi \rvert \le \max\, (1, \sqrt 2 B),
\end{cases}
\end{align*}
and $\varphi_M (\xi) = 1 - \varphi_1 (\xi) - \varphi_\infty (\xi)$, respectively. Using these cut-off functions, we decompose $T(t)$ into the low, medium, and high frequencies in the Fourier space:
\begin{equation}
\begin{split}
& T (t) = T_1 (t) + T_M (t) + T_\infty (t), \quad T_1 (t) = {}^\top\! (T_{\phi, 1} (t), T_{\boldsymbol{u}, 1} (t)), \\
& T_M (t) = {}^\top\! (T_{\phi, M} (t), T_{\boldsymbol{u}, M} (t)), \quad T_\infty (t) = {}^\top\! (T_{\phi, \infty} (t), T_{\boldsymbol{u}, \infty} (t)), \\
& T_{\phi, 1} (t) \begin{pmatrix}
\phi_0 (x) \\ \bu_0 (x)
\end{pmatrix}
= \CF^{- 1} [\varphi_1 (\xi) \wh \phi (\xi, t)] (x), \\
& T_{\boldsymbol{u}, 1} (t) \begin{pmatrix}
\phi_0 (x) \\ \bu_0 (x)
\end{pmatrix}
= \CF^{- 1} [\varphi_1 (\xi) \wh \bu (\xi, t)] (x), \\
& T_{\phi, M} (t) \begin{pmatrix}
\phi_0 (x) \\ \bu_0 (x)
\end{pmatrix}
= \CF^{- 1} [\varphi_M (\xi) \wh \phi (\xi, t)] (x), \\
& T_{\boldsymbol{u}, M} (t) \begin{pmatrix}
\phi_0 (x) \\ \bu_0 (x)
\end{pmatrix}
= \CF^{- 1} [\varphi_M (\xi) \wh \bu (\xi, t)] (x), \\
& T_{\phi, \infty} (t) \begin{pmatrix}
\phi_0 (x) \\ \bu_0 (x)
\end{pmatrix}
= \CF^{- 1} [\varphi_\infty (\xi) \wh \phi (\xi, t)] (x), \\
& T_{\boldsymbol{u}, \infty} (t) \begin{pmatrix}
\phi_0 (x) \\ \bu_0 (x)
\end{pmatrix}
= \CF^{- 1} [\varphi_\infty (\xi) \wh \bu (\xi, t)] (x).
\end{split}
\end{equation}
\par
To show the estimates for the operator $T_\ell (t)$ $(\ell = 1, \infty)$, we set
\begin{alignat*}4
g_1^{(\pm)} (\theta) & = \pm \mathrm{i} \sqrt{\gamma} \theta \left(\sqrt{1 - \frac{A^2 - \kappa}{\gamma} \theta}\right) & \quad & \text{for $0 < \theta < B$}, \\
g_\infty^{(\pm)} (\theta) & = \left(\sqrt{A^2 - \kappa} \pm \sqrt{A^2 - \kappa - \frac{\gamma}{\theta^2}} \right) \theta^2 & \quad & \text{for $\theta > B$,}
\end{alignat*}
so that the eigenvalues $\lambda_\pm (\xi)$ given in \eqref{226} can be written by
\begin{align*}
\lambda_\pm (\xi) = - K \lvert \xi \rvert^2 + g_\ell^{(\pm)} (\lvert \xi \rvert)
\end{align*}
for $\ell = 1, \infty$, where we have set
\begin{align*}
K := A \bigg(1 - \sqrt{1 - \frac{\kappa}{A^2}}\bigg) > 0.
\end{align*}
We emphasize that $K$ does not vanish for all $\xi \in \BR^n \setminus \{0\}$ such that $\lvert \xi \rvert \neq B$. We first give the $\L^q - \L^r$ estimate for $T_\ell (t)$.
\begin{lemm}
\label{lem-2.4.4}
Let $q$ and $r$ satisfy $1 \le q \le r \le \infty$. For all $a \in \BN_0^n$, $t > 0$, and $\bU_0 \in \L^q (\BR^n)^{n + 1}$, the estimates
\begin{align*}
\left\lVert \pd_x^a T_1 (t) \bU_0 \right\rVert_{\L^r (\BR^n)} & \le C t^{- \frac{n}{2} (\frac{1}{q} - \frac{1}{r})} \left\lVert \bU_0 \right\rVert_{\L^q (\BR^n)}, \\
\left\lVert \pd_x^a T_\infty (t) \bU_0 \right\rVert_{\L^r (\BR^n)} & \le C e^{- c t} t^{- \frac{n}{2} (\frac{1}{q} - \frac{1}{r}) - \frac{\lvert a \rvert}{2}} \left\lVert \bU_0 \right\rVert_{\L^q (\BR^n)}
\end{align*}
hold with some positive constants $c$ and $C$ independent of $t$ and $\bU_0$.		
\end{lemm}
\begin{proof}
Let $T_\ell^{(\pm)} (t)$ and $T_\ell^{(\mu)} (t)$ be the operators defined by
\begin{equation}
\label{242}
\begin{split}
T_\ell^{(\pm)} (t) \bU_0 (x) & = \CF^{-1} \left[e^{\lambda_\pm t} \varphi_\ell (\xi) G^{(\pm)} (\xi) \wh \bU_0 (\xi) \right] (x), \\
T_\ell^{(\mu)} (t) \bU_0 (x) & = \CF^{-1} \left[e^{- \mu \lvert \xi \rvert^2 t} \varphi_\ell (\xi) G^{(\mu)} (\xi) \wh \bU_0 (\xi) \right] (x)
\end{split}
\end{equation}
for $\ell = 1, \infty$ with function $G$ independent of $t$ and $\bU_0$ such that $\lvert \xi^a \pd_\xi^a G (\xi) \rvert \le C$ for any multi-index $a \in \BN_0^n$. According to the $\L^q - \L^r$ estimate for the heat semigroup and the Miklin-type Fourier multiplier theorem, we have
\begin{align*}
\left\lVert \pd_x^a T^{(\pm)}_\infty (t) \bU_0 \right\rVert_{\L^r (\BR^n)} & = \left\lVert \pd_x^a e^{\frac{K}{2} t \Delta} \CF^{-1} \left[e^{t \left(- \frac{K}{2} \lvert \xi \rvert^2 + g_\infty^{(\pm)} (\lvert \xi \rvert)\right)} \varphi_\infty (\xi) G^{(\pm)} (\xi) \wh \bU_0 (\xi) \right] \right\rVert_{\L^r (\BR^n)} \\
& \le C e^{- c t} t^{- \frac{n}{2} (\frac{1}{q} - \frac{1}{r}) - \frac{\lvert a \rvert}{2}} \lVert \bU_0 \rVert_{\L^q (\BR^n)}
\end{align*}
for all $t > 0$, $1 \le q \le r \le \infty$, $a \in \BN_0^n$, and $\bU_0 \in \L^q (\BR^n)$, where $c$ is some positive constant independent of $x$, $t$, and $\bU_0$. Analogously, we also obtain
\begin{align*}
\left\lVert \pd_x^a T^{(\mu)}_\infty (t) \bU_0 \right\rVert_{\L^r (\BR^n)} \le C e^{- c t} t^{- \frac{n}{2} (\frac{1}{q} - \frac{1}{r}) - \frac{\lvert a \rvert}{2}} \lVert \bU_0 \rVert_{\L^q (\BR^n)}
\end{align*}
for all $t > 0$, $1 \le q \le r \le \infty$, and $\bU_0 \in \L^q (\BR^n)$. If $\lvert \xi \rvert \neq B$, from \eqref{229}, we observe that $\lvert G^{(\pm)}_{J, K} (\xi) \rvert \le C_B$ for $J, K = 1, \dots, n + 1$. Hence, combined with~\eqref{228}, the required estimate has been shown. Here, to derive the estimate for $T_1 (t)$, we have used same argument above and additionally employed the Bernstein-type inequality.	
\end{proof}
We next show the space-time estimates for $T_\ell (t)$, $\ell = \{1, \infty\}$.
\begin{lemm}
\label{lem-2.4.6}
If $p$ and $q$ satisfy \eqref{cond-212}, there exists a positive constant $C$ independent of $t$ such that
\begin{align*}
\lVert T_1 (t) \bU_0 \rVert_{\W^{1, p} (\BR_+; \L^q (\BR^n))} & \le C \left\lVert \bU_0 \right\rVert_{\L^q (\BR^n)}, \\
\lVert (- \Delta) T_1 (t) \bU_0 \rVert_{\L^p (\BR_+; \L^q (\BR^n))} & \le C \left\lVert \bU_0 \right\rVert_{\L^q (\BR^n)}, \\
\lVert T_\infty (t) \bU_0 \rVert_{\W^{1, p} (\BR_+; \L^q (\BR^n))} & \le C \left\lVert \bU_0 \right\rVert_{\H^{1, q} (\BR^n)}, \\
\lVert (- \Delta) T_\infty (t) \bU_0 \rVert_{\L^p (\BR_+; \L^q (\BR^n))} & \le C \left\lVert \bU_0 \right\rVert_{\H^{1, q} (\BR^n)}
\end{align*}
for any $\bU_0 \in \H^{1, q} (\BR^n)$.	
\end{lemm}
\begin{proof}
Let $T^{(\pm)}_\ell$ and $T^{(\mu)}_\ell$ be the operators defined in \eqref{242}. Using Lemma~\ref{lem-MR} with $f = 0$, we have
\begin{align*}
\left\lVert T_\ell^{(\pm)} (t) \bU_0 \right\rVert_{\W^{1, p} (\BR_+; \L^q (\BR^n))} \le C \left\lVert \CF^{-1} \left[e^{t \left(- \frac{K}{2} \lvert \xi \rvert^2 + g_\ell^{(\pm)} (\lvert \xi \rvert)\right)} \varphi_\ell (\xi) G^{(\pm)} (\xi) \wh \bU_0 (\xi) \right] \right\rVert_{\B^{2 (1 - 1/p)}_{q, p} (\BR^n)}
\end{align*}
Since we have the continuous embedding $\H^{1, q} (\BR^n) \hookrightarrow \B^{2(1 - 1/p)}_{q, p} (\BR^n)$ under the condition~\eqref{cond-212}, by the Miklin-type Fourier multiplier theorem, we deduce that
\begin{align*}
\left\lVert \CF^{-1} \left[e^{t \left(- \frac{K}{2} \lvert \xi \rvert^2 + g_\ell^{(\pm)} (\lvert \xi \rvert)\right)} \varphi_\ell (\xi) G^{(\pm)} (\xi) \wh \bU_0 (\xi) \right] \right\rVert_{\B^{2 (1 - 1/p)}_{q, p} (\BR^n)} \le C \lVert \bU_0 \rVert_{\H^{1, q} (\BR^n)},
\end{align*}
which yields that
\begin{align*}
\left\lVert T_\ell^{(\pm)} (t) \bU_0 \right\rVert_{\W^{1, p} (\BR_+; \L^q (\BR^n))} \le C \lVert \bU_0 \rVert_{\H^{1, q} (\BR^n)}.
\end{align*}
Similarly, we have
\begin{align*}
\left\lVert T_\ell^{(\nu)} (t) \bU_0 \right\rVert_{\W^{1, p} (\BR_+; \L^q (\BR^n))} \le C \lVert \bU_0 \rVert_{\H^{1, q} (\BR^n)}.
\end{align*}
Noting \eqref{228} and using Lemma \ref{lem-MR} and the Bernstein-type inequality, the proof has been finished.	
\end{proof}

The following lemma will be used when we estimate the nonlinear terms.
\begin{lemm}
\label{lem-2.4.7}
Let $1 < p, q < \infty$ satisfy \eqref{cond-212} and let $\ell = 1, \infty$. Then, there exists a positive constant $C$ independent of $t$ such that the estimates
\begin{align*}
\left\lVert \int_0^t T_\ell (t - \tau) f (\tau) \,\mathrm{d}\tau \right\rVert_{\W^{1, p} (\BR_+; \L^q (\BR^n))} & \le C \lVert f \rVert_{\L^p (\BR_+; \L^q (\BR^n))}, \\
\left\lVert \int_0^t (- \Delta) T_\ell (t - \tau) f (\tau) \,\mathrm{d}\tau \right\rVert_{\L^p (\BR_+; \L^q (\BR^n))} & \le C \lVert f \rVert_{\L^p (\BR_+; \L^q (\BR^n))}
\end{align*}	
hold for all $f \in \L^p (\BR_+; \L^q (\BR^n))$.		
\end{lemm}

\begin{proof}
Applying Lemma \ref{lem-MR} with $u_0 = 0$ and the similar argument as in the proof of Lemma \ref{lem-2.4.6}, the proof is completed.	
\end{proof}

We finally focus on $T_M (t)$. The following is the $\L^q - \L^r$ mapping property for $T_M (t)$.
\begin{lemm}
Let $1 \le q \le r \le \infty$ and $a \in \BN_0^n$. For any $t > 0$ and $\bU_0 \in \L^q (\BR^n)$, we have
\begin{align*}
\left\lVert \pd_x^a T_M (t) \bU_0 \right\rVert_{\L^r (\BR^n)} \le C e^{- c t} \left\lVert \bU_0 \right\rVert_{\L^q (\BR^n)},
\end{align*}
where $C$ and $c$ are some positive constants independent of $t$.		
\end{lemm}

\begin{proof}
The required estimate immediately follows from the $\L^q - \L^q$ decay estimates for $T_M (t)$ with $1 \le q \le \infty$. Indeed, since $\CF [T_M (t) f]$ is a compactly supported function, we can adopt the Bernstein-type inequality and the inclusion $\L^q (\BR^n) \hookrightarrow \L^r (\BR^n)$ for $1 \le r \le q \le \infty$. As for this embedding property, we refer to Wang \textit{et al}.~\cite[Proposition~1.16]{WHHG}. \par
Set
\begin{align*}
N_{1, 1} (x, t) & = \frac{1}{2 \pi \mathrm{i}} \CF^{- 1} \left[\left(\oint_\Gamma \frac{(z + 2 A \lvert \xi \rvert^2) e^{z t}}{\det (z)} \,\mathrm{d}z \right) \varphi_M (\xi) \right], \\
N_{1, k + 1} (x, t) & = \frac{1}{2 \pi \mathrm{i}} \CF^{- 1} \left[\mathrm{i} \xi_k \left(\oint_\Gamma \frac{e^{z  t}}{\det (z)} \,\mathrm{d}z \right) \varphi_M (\xi) \right], \\
N_{j + 1, 1} (x, t) & = - \frac{1}{2 \pi} \CF^{- 1} \left[(\gamma + \kappa \lvert \xi \rvert^2)\left(\oint_\Gamma \frac{e^{z t}}{\det (z)} \,\mathrm{d}z \right) \varphi_M (\xi) \right], \\
N_{j + 1, k + 1} (x, t) & = \frac{1}{2 \pi \mathrm{i}} \CF^{- 1} \left[\frac{\xi_j \xi_k}{\lvert \xi \rvert^2} \left(\oint_\Gamma \frac{z e^{z t}}{\det (z)} \,\mathrm{d}z \right) \varphi_M (\xi) \right], \\
\wt N (x, t) & = \CF^{- 1} \left[e^{- \mu \lvert \xi \rvert^2 t} \varphi_M (\xi) \right], \\
\wt N_{j, k} (x, t) & = \CF^{- 1} \left[\frac{\xi_j \xi_k}{\lvert \xi \rvert^2} e^{- \mu \lvert \xi \rvert^2 t} \varphi_M (\xi) \right].
\end{align*}
for $j, k = 1, \dots, n$. Recalling \eqref{2210}, we have the following formulas:
\begin{equation}
\label{243}
\begin{split}
\CF^{- 1} \Big[\varphi_M (\xi) \wh \phi \Big] (x, t) & = N_{1, 1} (\cdot, t) * \phi_0 + \sum_{k = 1}^n N_{1, k + 1} (t, \cdot) * u_{0, k}, \\
\CF^{- 1} \Big[\varphi_M (\xi) \wh u_j \Big] (x, t) & = N_{j + 1, 1} (\cdot, t) * \phi_0 + \wt N (\cdot, t) * u_{0, j} \\
& \quad + \sum_{k = 1}^n \Big(N_{j + 1, k + 1} (\cdot, t) + \wt N_{j, k} (\cdot, t) \Big)* u_{0, k} 
\end{split}
\end{equation}
for $j = 1, \dots, n$. Notice that the Fourier transforms of $N_{J, K} (x, t)$ $(J, K = 1, \dots, n + 1)$, $\wt N (x, t)$, and $\wt N_{j, k} (x, t)$ are compactly supported functions due to the cut-off function $\varphi_M (\xi)$. Thus, from the residue theorem and the condition \eqref{2211}, we easily see that
\begin{align*}
\lvert N_{J, K} (x, t) \rvert \le C e^{- c_0 t}
\end{align*}
holds for $J, K = 1, \dots, n + 1$ with some positive constant $C$ independent of $x$ and $t$, where $c_0$ is the same constant given in \eqref{2211}. On the other hand, we also have
\begin{align*}
\lvert \wt N_{j, k} (x, t) \rvert \le C e^{- c_1 t} \quad \text{for $j, k = 1, \dots, n$} \quad \text{and} \quad \lvert \wt N (x, t) \rvert \le C e^{- c_1 t},
\end{align*}
where $C$ is some positive constant independent of $t$ and $x$ and $c_1$ is some positive constant depending only on $\mu$ and $B$. Hence, applying the Young inequality to~\eqref{243}, we obtain
\begin{align*}
\lVert T_{\phi, M} (t) \bU_0 \rVert_{\L^q (\BR^n)} & \le C_{A, B, n, q} e^{- c t} \lVert \bU_0 \rVert_{\L^q (\BR^n)}, \\
\lVert T_{\bu, M} (t) \bU_0 \rVert_{\L^q (\BR^n)} & \le C_{A, B, n, q} e^{- c t} \lVert \bU_0 \rVert_{\L^q (\BR^n)}
\end{align*}
for $1 \le q \le \infty$ and $\bU_0 \in \L^q (\BR^n)$ with $c := \min (c_0, c_1)$. Namely, we obtain
\begin{align*}
\lVert T_M (t) \bU_0 \rVert_{\L^q (\BR^n)} \le C_{A, B, n, q} e^{- c t} \lVert \bU_0 \rVert_{\L^q (\BR^n)}
\end{align*}
for all $1 \le q \le \infty$ and $\bU_0 \in \L^q (\BR^n)$.		
\end{proof}

The following lemmas can be proved by employing the same arguments as in the proof of Lemmas~\ref{lem-2.4.6} and \ref{lem-2.4.7}, so that we may omit the proofs.
\begin{lemm}
Let $(p, q)$ satisfy \eqref{cond-212}. Then the estimates
\begin{align*}
\lVert T_M (t) \bU_0 \rVert_{\W^{1, p} (\BR_+; \L^q (\BR^n))} & \le C \lVert \bU_0 \rVert_{\H^{1, q} (\BR^n)}, \\
\lVert (- \Delta) T_M (t) \bU_0 \rVert_{\L^p (\BR_+; \L^q (\BR^n))} & \le C \left\lVert \bU_0 \right\rVert_{\H^{1, q} (\BR^n)},
\end{align*}
hold for all $\bU_0 \in \H^{1, q} (\BR^n)$ with some positive constant $C$ independent of $t$.	
\end{lemm}

\begin{lemm}
\label{lem-2410}
Let $p$ and $q$ satisfy \eqref{cond-212}. There exists a positive constant $C$ independent of $t$ such that the estimates
\begin{align*}
\left\lVert \int_0^t T_M (t - \tau) f (\tau) \,\mathrm{d}\tau \right\rVert_{\W^{1, p} (\BR_+; \L^q (\BR^n))} & \le C \lVert f \rVert_{\L^p (\BR_+; \L^q (\BR^n))}, \\
\left\lVert \int_0^t (- \Delta) T_M (t - \tau) f (\tau) \,\mathrm{d}\tau \right\rVert_{\L^p (\BR_+; \L^q (\BR^n))} & \le C \lVert f \rVert_{\L^p (\BR_+; \L^q (\BR^n))}
\end{align*}	
hold true for all $f \in \L^p (\BR_+; \L^q (\BR^n))$.		
\end{lemm}

Recalling that $T (t) = T_1 (t) + T_M (t) + T_\infty (t)$, we see that Lemmas \ref{lem-2.4.4}--\ref{lem-2410} imply Lemmas~\ref{lem-2.4.1}--\ref{lem-2.4.3}, so that we may omit the proof of Lemmas~\ref{lem-2.4.1}--\ref{lem-2.4.3} for Case 1. \par
We next deal with Case 2,3,4,6. To simplify the notation, we introduce functions $g^{(\pm)}_m (\theta)$, $m = 2,3,4,6$, defined by
\begin{alignat*}4
g^{(\pm)}_2 (\theta) & = \pm \mathrm{i} A \lvert \xi \rvert \sqrt{\bigg(\frac{\kappa}{A^2} - 1\bigg) \lvert \xi \rvert^2 + \frac{\gamma}{A}} & \quad & \text{for Case 2}, \\
g^{(\pm)}_3 (\theta) & = \pm \mathrm{i} \frac{\sqrt{\gamma}}{A} \lvert \xi \rvert & \quad & \text{for Case 3}, \\
g^{(+)}_4 (\theta) & = 0 & \quad & \text{for Case 4}, \\
g^{(-)}_4 (\theta) & = - 2 A \sqrt{1 - \frac{\kappa}{A^2}} \lvert \xi \rvert^2 & \quad & \text{for Case 4}, \\
%g^{(\pm)}_5 (\theta) & = \pm \mathrm{i} \left(\sqrt{\frac{\kappa}{A^2} - 1}\right) \lvert \xi \rvert^2 & \quad & \text{for Case 5}, \\
g^{(\pm)}_6 (\theta) & = 0 & \quad & \text{for Case 6},
\intertext{where $0 < \theta < \infty$, and constants $K_m$ defined by}
K_2 & = A & \quad & \text{for Case 2}, \\
K_3 & = A & \quad & \text{for Case 3}, \\
K_4 & = A\bigg(1 - \sqrt{1 - \frac{\kappa}{A^2}} \bigg) & \quad & \text{for Case 4}, \\
%K_5 & = A & \quad & \text{for Case 5}, \\
K_6 & = A & \quad & \text{for Case 6}.
\end{alignat*}
We see that $\lambda_\pm (\xi)$ can be written in the form of $\lambda_\pm (\xi) = - K_m \lvert \xi \rvert^2 + g^{(\pm)}_m (\lvert \xi \rvert)$ for each Case $m = 2,3,4,6$, where $K_m$ does not vanish for all $\xi \in \BR^n \setminus \{0\}$. Accordingly, applying the argument for the proof of Lemmas~\ref{lem-2.4.4}--\ref{lem-2.4.7}, we immediately have the estimates of $T (t)$ for Case~2,3,4,6. Summing up, we obtain Lemmas \ref{lem-2.4.1}--\ref{lem-2.4.3} for Case 1--6 exclude Case 5.

\section{Nonlinear problem}
\label{subsec-2.5}
\noindent
Finally, we prove Theorem \ref{Th-2.1.1}. As we mentioned before, we can prove the existence of global-in-time solution via the standard Picard fixed point iteration because we have the parabolic smoothing not only for the momentum but for the density, which is different from the case for the ``usual" compressible Navier-Stokes equation---this fact is a consequence of the capillary effects on the fluids. \par
From the Duhamel principle, the equations \eqref{eq-221} can be transformed to the integral equation:
\begin{align}
\label{251}
\bU (t) = T(t) \bU_0 - \int_0^t T(t - \tau) \bN (\tau) \,\mathrm{d}\tau
\end{align}
for $t \ge 0$, where $\bN (\tau) := {}^\top\! (0, \bF (\pi (x, \tau), \bmm (x, \tau)))$. In what follows, let $p$, $q$, and $s$ satisfy \eqref{cond-212} and~\eqref{cond-s}. Then we define the underlying space $X_{p, q, s}$ by
\begin{align*}
X_{p, q, s} := \left\{\bU \in Y_{p, q, s} \mid \lVert \bU \rVert_{X_{p, q, s}} \le C_0 \lVert \bU_0 \rVert_{\H^{s + 1, q} (\BR^n) \times \H^{s, q} (\BR^n)} \right\}
\end{align*}
endowed with the norm
\begin{align*}
\lVert \bU \rVert_{X_{p, q, s}} & := \lVert \pi \rVert_{\W^{1, p} (\BR_+; \H^{s, q} (\BR^n))} + \lVert \pi \rVert_{\L^p (\BR_+; \H^{s + 2, q} (\BR^n))} \\
& \quad + \lVert \bmm \rVert_{\W^{1, p} (\BR_+; \H^{s - 1, q} (\BR^n))} + \lVert \bmm \rVert_{\L^p (\BR_+; \H^{s + 1, q} (\BR^n))}.
\end{align*}
Here, we have set
\begin{gather*}
\bU \in Y_{p, q, s} \Longleftrightarrow 
\left\{\begin{aligned}
\pi \in \W^{1, p} (\BR_+; \H^{s, q} (\BR^n)) \cap \L^p (\BR_+; \H^{s + 2, q} (\BR^n)), \\
\bmm \in \W^{1, p} (\BR_+; \H^{s - 1, q} (\BR^n)) \cap \L^p (\BR_+; \H^{s + 1, q} (\BR^n)),
\end{aligned} \right. \\
\lVert \bU_0 \rVert_{\H^{s + 1, q} (\BR^n) \times \H^{s, q} (\BR^n)} := \lVert \pi_0 \rVert_{\H^{s + 1, q} (\BR^n)} + \lVert \bmm_0 \rVert_{\H^{s, q} (\BR^n)}.
\end{gather*}
Using the norm $\lVert \, \cdot \, \rVert_{X_{p, q, s}}$, from Lemma \ref{lem-2.4.2}, there exists a positive constant $C_0$ independent of $t$ and $\bU_0$ such that
\begin{align}
\label{252}
\lVert T(t) \bU_0 \rVert_{X_{p, q, s}} \le C_0 \lVert \bU_0 \rVert_{\H^{s + 1, q} (\BR^n) \times \H^{s, q} (\BR^n)}
\end{align}
for any $\bU_0 \in \H^{s + 1, q} (\BR^n) \times \H^{s, q} (\BR^n)^n$. To construct a solution to \eqref{251}, we have to estimate the nonlinear terms.
\begin{lemm}
\label{lem-2.5.1}
Let $p$, $q$, and $s$ satisfy \eqref{cond-212} and \eqref{cond-s}. Define $\bN_1 (\tau)$ and $\bN_2 (\tau)$ by
\begin{align*}
\bN_1 (\tau) := {}^\top\! (0, \bF (\pi_1 (x, \tau), \bmm_1 (x, \tau))), \quad \bN_2 (\tau) := {}^\top\! (0, \bF (\pi_2 (x, \tau), \bmm_2 (x, \tau)))
\end{align*}
for $\tau \in [0, t]$ with $t > 0$. Assume that
\begin{align}
\label{253}
C_0 \lVert \bU_0 \rVert_{\H^{s + 1, q} (\BR^n) \times \H^{s, q} (\BR^n)} & \le \frac{1}{2},
\end{align}
where $C_0$ is a constant satisfying \eqref{252}. Then, for all $t > 0$ and $\bU, \bU_1, \bU_2 \in X_{p, q, s}$, we have the estimates
\begin{align}
\label{254}
\bigg\lVert \int_0^t T(t - \tau) \bN (\tau) \,\mathrm{d}\tau \bigg\rVert_{X_{p, q, s}} & \le C_1 \lVert \bU_0 \rVert_{\H^{s + 1, q} (\BR^n) \times \H^{s, q} (\BR^n)}^2, \\
\label{255}
\bigg\lVert \int_0^t T(t - \tau) (\bN_1 (\tau) - \bN_2 (\tau))\,\mathrm{d}\tau \bigg\rVert_{X_{p, q, s}} & \le C_2 \lVert \bU_0 \rVert_{\H^{s + 1, q} (\BR^n) \times \H^{s, q} (\BR^n)} \lVert \bU_1 - \bU_2 \rVert_{X_{p, q, s}},
\end{align}
with some positive constants $C_1$ and $C_2$ depending on $C_0$, where $\bU_1 := {}^\top\! (\pi_1, \bmm_1)$ and $\bU_2 := {}^\top\! (\pi_2, \bmm_2)$.	
\end{lemm}

We here introduce the estimates for composition operators to bound the nonlinear terms including pressure term.
\begin{lemm}
\label{lem-2.5.2}	
Let $1 < q < \infty$ and $0 \in I \subset \BR$. Let $s$ satisfy \eqref{cond-s}. Furthermore, let $f \colon I \to \BR$ be a smooth~(at least $\C^{[s] + 1}$-class) function such that $f (0) = 0$. The following assertions hold true.
\begin{enumerate}\renewcommand{\labelenumi}{(\arabic{enumi})}
\item If $u \in \H^{s, q} (\BR^n)$, then the composition operator $f (u)$ belongs to $\H^{s, q} (\BR^n)$ bounded by
\begin{align*}
\lVert f (u) \rVert_{\H^{s, q} (\BR^n)} \le C \Big(\lVert u \rVert_{\H^{s, q} (\BR^n)} + \lVert u \rVert_{\H^{s, q} (\BR^n)}^s \Big)
\end{align*}
with some positive constant $C$ depending only on $f', \dots, f^{([s])}$, $n$, $q$, and $s$.
\item If, in addition, $f$ satisfies $f'(0) = 0$, then for any $u, v \in \H^{s, q} (\BR^n)$, the difference $f (u) - f (v)$ is the element of $\H^{s, q} (\BR^n)$ possessing the estimate:
\begin{align*}
\lVert f (u) - f(v) \rVert_{\H^{s, q} (\BR^n)} & \le C \Big(\lVert u - v \rVert_{\H^{s, q} (\BR^n)} \sup_{\tau \in [0, 1]} \lVert u + \tau (v - u) \rVert_{\H^{s, q} (\BR^n)} \\
& \quad + \lVert u - v \rVert_{\H^{s, q} (\BR^n)} \sup_{\tau \in [0, 1]} \lVert u + \tau (v - u) \rVert_{\H^{s, q} (\BR^n)}^s \Big),
\end{align*}
where $C$ is some constant depending only on $f'', \dots, f^{([s] + 1)}$, $n$, $q$, and $s$.
\end{enumerate}		
\end{lemm}
\begin{rema}
When $0 < s \le 1$, it holds
\begin{align*}
\lVert f (u) \rVert_{\H^{s, q} (\BR^n)} \le C K \lVert u \rVert_{\H^{s, q} (\BR^n)}
\end{align*}
with a positive constant $C$, where $1 < q < \infty$, $f (0) = 0$, and $\lvert f' \rvert \le K$. One can find the proof in Christ and Weinstein~\cite{CW91} and Taylor~\cite[Chapter 2]{T00}.
\end{rema}
\begin{proof}[Proof of Lemma~\ref{lem-2.5.2}]
We first extend the domain of $f$ to $\BR$ by zero. According to Adams and Frazier~\cite{AF92}, we have the estimate
\begin{align*}
\lVert f(u) \rVert_{\H^{s, q} (\BR^n)} & \le C \max_{k \in \{1, \dots, [s]\} } \sup_{\theta \in \BR} \left\lvert f^{(k)} (\theta) \right\rvert \Big(\lVert u \rVert_{\H^{s, q} (\BR^n)} + \lVert u \rVert^s_{\dot \H^{1, s q} (\BR^n)} \Big) \\
& \le C \max_{k \in \{1, \dots, [s]\} } \sup_{\theta \in \BR} \left\lvert f^{(k)} (\theta) \right\rvert \Big(\lVert u \rVert_{\H^{s, q} (\BR^n)} + \lVert u \rVert_{\H^{s, q} (\BR^n)}^s \Big),
\end{align*}
where we have employed the embedding $\H^{s, q} (\BR^n) \hookrightarrow \H^{1, s q} (\BR^n)$. Notice that this embedding holds for any $q \in (1, \infty)$ whenever $s \ge 1$. Here, the constant $C$ is independent of $f$ and $u$. In addition, we know the identity
\begin{align*}
f(u) - f (v) = (v - u) \int_0^t f'(u + \tau (v - u)) \,\mathrm{d}\tau,
\end{align*}
which concludes the proof.
\end{proof}
We also need the estimates for composition operators with the Besov norm. The proof of the following lemma can be found in the book of Bahouri \textit{et al}.~\cite{BCD}, see Theorem 2.87 and Corollary~2.91.
\begin{lemm}
\label{lem-2.5.4}
Let $1 \le p, q \le \infty$, $s > 0$, and $0 \in I \subset \BR$. Furthermore, let $f \colon I \to \BR$ be a smooth (at least $\C^{[s] + 1}$-class) function such that $f (0) = 0$. The following statements are valid.
\begin{enumerate}\renewcommand{\labelenumi}{(\arabic{enumi})}
\item If $u \in \B^s_{q, p} (\BR^n) \cap \L^\infty (\BR^n)$, then we have $f (u) \in \B^s_{q, p} (\BR^n) \cap \L^\infty (\BR^n)$ bounded by
\begin{align*}
\lVert f (u) \rVert_{\B^s_{q, p} (\BR^n)} \le C \lVert u \rVert_{\B^s_{q, p} (\BR^n)}
\end{align*}
with some positive constant $C$ depending only on $f'$, $s$, and $\lVert u \rVert_{\L^\infty (\BR^n)}$.
\item If, furthermore, $f$ satisfies $f'(0) = 0$, then for any $u, v \in \B^s_{q, p} (\BR^n) \cap \L^\infty (\BR^n)$, the difference $f (u) - f (v)$ also belongs to $\B^s_{q, p} (\BR^n) \cap \L^\infty (\BR^n)$ with the estimate:
\begin{align*}
\lVert f (u) - f(v) \rVert_{\H^{s, q} (\BR^n)} & \le C \Big(\lVert u - v \rVert_{\B^s_{q, p} (\BR^n)} \sup_{\tau \in [0, 1]} \lVert u + \tau (v - u) \rVert_{\L^\infty (\BR^n)} \\
& \quad + \lVert u - v \rVert_{\L^\infty (\BR^n)} \sup_{\tau \in [0, 1]} \lVert u + \tau (v - u) \rVert_{\B^s_{q, p} (\BR^n)} \Big),
\end{align*}
where $C$ is some constant depending on $f''$, $\lVert u \rVert_{\L^\infty (\BR^n)}$, and $\lVert v \rVert_{\L^\infty (\BR^n)}$.	
\end{enumerate}	
\end{lemm}

To estimate the nonlinear terms, we will use the \textit{Kato-Ponce inequality}~\cite{KP88}, see also Sawano~\cite[Theorem 4.44]{Saw18}.
\begin{lemm}
\label{lem-bilinear}
Let $1 < q < \infty$ and $s > 0$. For $f, g \in \H^{s, q} (\BR^n) \cap \L^\infty (\BR^n)$, we have
\begin{align*}
\lVert f g \rVert_{\H^{s, q} (\BR^n)} \le C \Big(\lVert f \rVert_{\H^{s, q} (\BR^n)} \lVert g \rVert_{\L^\infty (\BR^n)} + \lVert f \rVert_{\L^\infty (\BR^n)} \lVert g \rVert_{\H^{s, q} (\BR^n)}\Big)
\end{align*}
with a positive constant $C$.
\end{lemm}

\begin{proof}[Proof of Lemma \ref{lem-2.5.1}]
To show \eqref{254}, it suffices to prove the bounds
\begin{align}
\label{256}
\bigg\lVert \int_0^t T(t - \tau) \bF (\tau) \,\mathrm{d}\tau \bigg\rVert_{\L^p (\BR_+; \H^{s + 1, q} (\BR^n))} \le C_1 \lVert \bU_0 \rVert_{\H^{s + 1, q} (\BR^n) \times \H^{s, q} (\BR^n)}^2, \\
\label{257}
\bigg\lVert \int_0^t T(t - \tau) \bF (\tau) \,\mathrm{d}\tau \bigg\rVert_{\W^{1, p} (\BR_+; \H^{s - 1, q} (\BR^n))} \le C_1 \lVert \bU_0 \rVert_{\H^{s + 1, q} (\BR^n) \times \H^{s, q} (\BR^n)}^2.
\end{align}
Using the embedding property \eqref{emb-MR} and the assumption \eqref{253}, for $\bU \in X_{p, q, s}$ we have the bound
\begin{equation}
\label{258}
\begin{split}
\sup_{t \in \BR_+} \lVert \pi \rVert_{\H^{s, q} (\BR^n)} & \le C \sup_{t \in \BR_+} \lVert \pi \rVert_{\B^{s + 2 - 2/p}_{q, p} (\BR^n)} \\
& \le C \Big(\lVert \pi \rVert_{\W^{1, p} (\BR_+; \H^{s, q} (\BR^n))} + \lVert \pi \rVert_{\L^p (\BR_+; \H^{s + 2, q} (\BR^n))} \Big) \\
& \le C C_0 \lVert \bU_0 \rVert_{\H^{s + 1, q} (\BR^n) \times \H^{s, q} (\BR^n)} \\
& \le \frac{C}{2} < \infty,
\end{split}
\end{equation}
which, combined with Lemma \ref{lem-2.5.2}, yields that
\begin{align*}
\lVert G (\pi) \rVert_{\H^{s, q} (\BR^n)} \le C \left(\lVert \pi \rVert_{\H^{s, q} (\BR^n)} + \lVert \pi \rVert_{\H^{s, q} (\BR^n)}^s \right).
\end{align*}
Besides, employing Lemma \ref{lem-2.5.4}, we observe
\begin{align*}
\lVert G (\pi) \rVert_{\B^{s + 2 - 2/p}_{q, p} (\BR^n)} \le C \lVert \pi \rVert_{\B^{s + 2 - 2/p}_{q, p} (\BR^n)},
\end{align*}
where the constant may depend on $C_0$. Using Lemma~\ref{lem-bilinear}, we obtain
\begin{align*}
\lVert G (\pi) \nabla \pi \rVert_{\H^{s - 1, q} (\BR^n)} & \le C \Big(\lVert G (\pi) \rVert_{\H^{s - 1, q} (\BR^n)} \lVert \nabla \pi \rVert_{\L^\infty (\BR^n)} + \lVert G (\pi) \rVert_{\L^\infty (\BR^n)} \lVert \nabla \pi \rVert_{\H^{s - 1, q} (\BR^n)}\Big).
\end{align*}
Then, from Lemma \ref{lem-2.4.3}, we have the estimate
{\allowdisplaybreaks
\begin{align*}
& \bigg\lVert \int_0^t T (t - \tau) \Big(G (\pi) \nabla \pi \Big) \,\mathrm{d}\tau \bigg\rVert_{\L^p (\BR_+; \H^{s + 1, q} (\BR^n))} \\
& \quad \le C \lVert G (\pi) \nabla \pi \rVert_{\L^p (\BR_+; \H^{s - 1, q} (\BR^n))} \\
& \quad \le C \Big(\lVert G (\pi) \rVert_{\L^p(\BR_+; \H^{s - 1, q} (\BR^n))} \lVert \nabla \pi \rVert_{\L^\infty (\BR_+; \L^\infty (\BR^n))} \\
& \quad \quad + \lVert G (\pi) \rVert_{\L^\infty (\BR_+; \L^\infty (\BR^n))} \lVert \nabla \pi \rVert_{\L^p (\BR_+; \H^{s - 1, q} (\BR^n))}\Big) \\
& \quad \le C \Big(\lVert G(\pi) \rVert_{\L^p(\BR_+; \H^{s - 1, q} (\BR^n))} \lVert \nabla \pi \rVert_{\L^\infty (\BR_+; \B^{s + 1 - 2/p}_{q, p} (\BR^n))} \\
& \quad \quad + \lVert G (\pi) \rVert_{\L^\infty (\BR_+; \B^{s + 2 - 2/p}_{q, p} (\BR^n))} \lVert \nabla \pi \rVert_{\L^p (\BR_+; \H^{s - 1, q} (\BR^n))}\Big) \\
& \quad \le C \Big\{\Big(\lVert \pi \rVert_{\L^p(\BR_+; \H^{s, q} (\BR^n))} + \lVert \pi \rVert^s_{\L^p (\BR_+; \H^{s, q} (\BR^n))}\Big) \\
& \quad \quad \quad \cdot \Big(\lVert \pi \rVert_{\W^{1, p} (\BR_+; \H^{s, q} (\BR^n))} + \lVert \pi \rVert_{\L^p (\BR_+; \H^{s + 2, q} (\BR^n))} \Big) \\
& \quad \quad + \Big(\lVert \pi \rVert_{\L^p(\BR_+; \H^{s, q} (\BR^n))} + \lVert \pi \rVert^s_{\L^p (\BR_+; \H^{s, q} (\BR^n))}\Big) \lVert \pi \rVert_{\L^p (\BR_+; \H^{s, q} (\BR^n))}\Big\} \\
& \quad \le C \lVert \bU_0 \rVert_{\H^{s + 1, q} (\BR^n) \times \H^{s, q} (\BR^n)}^2
\end{align*}}\noindent
for $\bU \in X_{p, q, s}$, where we have used the embedding properties:
\begin{gather*}
\B^{s + 2}_{q, p} (\BR^n) \hookrightarrow \B^{s + 1}_{q, p} (\BR^n) \hookrightarrow \L^\infty (\BR^n), \\
\W^{1, p} (\BR_+; \H^{s, q} (\BR^n)) \cap \L^p (\BR_+; \H^{s + 2, q} (\BR^n)) \hookrightarrow \L^\infty (\BR_+; \B^{s + 2 - 2/p}_{q, p} (\BR^n))
\end{gather*}	
and the estimate
\begin{equation}
\label{est-pi}
\begin{split}
& \lVert \pi \rVert_{\L^p(\BR_+; \H^{s, q} (\BR^n))} + \lVert \pi \rVert^s_{\L^p (\BR_+; \H^{s, q} (\BR^n))} \\
& \le \Big(1 + \lVert \pi \rVert_{\L^p(\BR_+; \H^{s, q} (\BR^n))}^{s - 1}\Big) \lVert \pi \rVert_{\L^p(\BR_+; \H^{s, q} (\BR^n))} \\
& \le \Big(1 + \lVert \bU \rVert_{X_{p, q, s}}^{s - 1}\Big) \lVert \pi \rVert_{\L^p(\BR_+; \H^{s, q} (\BR^n))} \\
& \le \Big(1 + C_0^{s - 1} \lVert \bU_0 \rVert_{\H^{s + 1, q} (\BR^n) \times \H^{s, q} (\BR^n)}^{s - 1}\Big) \lVert \pi \rVert_{\L^p(\BR_+; \H^{s, q} (\BR^n))} \\
& \le \bigg(1 + \frac{1}{2^{s - 1}}\bigg) \lVert \pi \rVert_{\L^p(\BR_+; \H^{s, q} (\BR^n))}.
\end{split}
\end{equation}
Similarly, we also obtain
\begin{align*}
\bigg\lVert \int_0^t T (t - \tau) \Big(P''(1) \pi \nabla \pi \Big) \,\mathrm{d}\tau \bigg\rVert_{\L^p (\BR_+; \H^{s + 1, q} (\BR^n))} \le C \lVert \bU_0 \rVert_{\H^{s + 1, q} (\BR^n) \times \H^{s, q} (\BR^n)}^2.
\end{align*}
Since we see that
\begin{align*}
\left\lVert \frac{1}{1 + \pi} \right\rVert_{\L^p (\BR_+; \H^{s + 2, q} (\BR^n))} \le \frac{1}{1 - C_0 \lVert \bU_0 \rVert_{\H^{s + 1, q} (\BR^n) \times \H^{s, q} (\BR^n)}} \le 2, \\
\left\lVert \frac{1}{1 + \pi} \right\rVert_{\L^\infty (\BR_+; \L^\infty (\BR^n))} \le \frac{1}{1 - C_0 \lVert \bU_0 \rVert_{\H^{s + 1, q} (\BR^n) \times \H^{s, q} (\BR^n)}} \le 2
\end{align*}
hold for any $\bU \in X_{p, q, s}$ supposing \eqref{253}, the Sobolev embedding theorem and Lemma \ref{lem-2.4.3} implies	
{\allowdisplaybreaks \begin{align*}
& \bigg\lVert \int_0^t T (t - \tau) \dv \bigg(\frac{1}{1 + \pi (\tau)}\bmm (\tau)\otimes \bmm (\tau) \bigg) \,\mathrm{d}\tau \bigg\rVert_{\L^p (\BR_+; \H^{s + 1, q} (\BR^n))} \\
& \le C \left\lVert \frac{1}{1 + \pi} \bmm \otimes \bmm \right\rVert_{\L^p (\BR_+; \H^{s - 1, q} (\BR^n))} \\
& \le C \left(\left\lVert \frac{1}{1 + \pi} \right\rVert_{\L^p (\BR_+; \H^{s - 1, q} (\BR^n))} \lVert \bmm \rVert_{\L^\infty (\BR_+; \L^\infty (\BR^n))}^2 \right. \\
& \left. \quad + 2 \left\lVert \frac{1}{1 + \pi} \right\rVert_{\L^\infty (\BR_+; \L^\infty (\BR^n))} \lVert \bmm \rVert_{\L^p (\BR_+; \H^{s - 1, q} (\BR^n))} \lVert \bmm \rVert_{\L^\infty (\BR_+; \L^\infty (\BR^n))} \right) \\
& \le C \left(\left\lVert \frac{1}{1 + \pi} \right\rVert_{\L^p (\BR_+; \H^{s - 1, q} (\BR^n))} \lVert \bmm \rVert_{\L^\infty (\BR_+; \B^{s + 1, q}_{q, p} (\BR^n))}^2 \right. \\
& \left. \quad + 2 \left\lVert \frac{1}{1 + \pi} \right\rVert_{\L^\infty (\BR_+; \B^{s + 2 - 2/p}_{q, p} (\BR^n))} \lVert \bmm \rVert_{\L^p (\BR_+; \H^{s - 1, q} (\BR^n))} \lVert \bmm \rVert_{\L^\infty (\BR_+; \B^{s + 1 - 2/p}_{q, p} (\BR^n))} \right) \\
& \le C \left(\Big(\lVert \bmm \rVert_{\W^{1, p} (\BR_+; \H^{s - 1, q} (\BR^n))} + \lVert \bmm \rVert_{\L^p (\BR_+; \H^{s + 1, q} (\BR^n))} \Big)^2 \right. \\
&  \left. \quad + \lVert \bmm \rVert_{\L^p (\BR_+; \H^{s - 1, q} (\BR^n))} \Big(\lVert \bmm \rVert_{\W^{1, p} (\BR_+; \H^{s - 1, q} (\BR^n))} + \lVert \bmm \rVert_{\L^p (\BR_+; \H^{s + 1, q} (\BR^n))} \Big) \right) \\
& \le C C_0^2 \lVert \bU_0 \rVert_{\H^{s + 1, q} (\BR^n) \times \H^{s, q} (\BR^n)}^2.
\end{align*}}\noindent	
Using the same argument, we also arrive at
{\allowdisplaybreaks
\begin{align*}
& \bigg\lVert \int_0^t T (t - \tau) \Delta \bigg(\frac{\pi (\tau)}{1 + \pi (\tau)}\bmm (\tau) \bigg) \,\mathrm{d}\tau \bigg\rVert_{\L^p (\BR_+; \H^{s + 1, q} (\BR^n))} \\
& \qquad \le C \left\lVert \frac{\pi}{1 + \pi} \bmm \right\rVert_{\L^p (\BR_+; \H^{s + 1, q} (\BR^n))} \\
& \qquad \le C \left( \left\lVert \frac{1}{1 + \pi} \right\rVert_{\L^p (\BR_+; \H^{s + 1, q} (\BR^n))} \lVert \pi \rVert_{\L^\infty (\BR_+; \L^\infty (\BR^n))} \lVert \bmm \rVert_{\L^\infty (\BR_+; \L^\infty (\BR^n))} \right. \\
& \qquad \quad + \left\lVert \frac{1}{1 + \pi} \right\rVert_{\L^\infty (\BR_+; \L^\infty (\BR^n))} \lVert \pi \rVert_{\L^p (\BR_+; \H^{s + 1, q} (\BR^n))} \lVert \bmm \rVert_{\L^\infty (\BR_+; \L^\infty (\BR^n))} \\
& \qquad \quad \left. + \left\lVert \frac{1}{1 + \pi} \right\rVert_{\L^\infty (\BR_+; \L^\infty (\BR^n))} \lVert \pi \rVert_{\L^\infty (\BR_+; \L^\infty (\BR^n))} \lVert \bmm \rVert_{\L^p (\BR_+; \H^{s + 1, q} (\BR^n))} \right)\\
& \qquad \le C C_0^2 \lVert \bU_0 \rVert_{\H^{s + 1, q} (\BR^n) \times \H^{s, q} (\BR^n)}^2 \\
& \bigg\lVert \int_0^t T (t - \tau) \nu \nabla \dv \bigg(\frac{\pi (\tau)}{1 + \pi (\tau)} \bmm (\tau) \bigg)\,\mathrm{d}\tau \bigg\rVert_{\L^p (\BR_+; \H^{s + 1, q} (\BR^n))} \\
& \qquad \le C \left\lVert \frac{\pi}{1 + \pi} \bmm \right\rVert_{\L^p (\BR_+; \H^{s + 1, q} (\BR^n))} \\
& \qquad \le C \left( \left\lVert \frac{1}{1 + \pi} \right\rVert_{\L^p (\BR_+; \H^{s + 1, q} (\BR^n))} \lVert \pi \rVert_{\L^\infty (\BR_+; \L^\infty (\BR^n))} \lVert \bmm \rVert_{\L^\infty (\BR_+; \L^\infty (\BR^n))} \right. \\
& \qquad \quad + \left\lVert \frac{1}{1 + \pi} \right\rVert_{\L^\infty (\BR_+; \L^\infty (\BR^n))} \lVert \pi \rVert_{\L^p (\BR_+; \H^{s + 1, q} (\BR^n))} \lVert \bmm \rVert_{\L^\infty (\BR_+; \L^\infty (\BR^n))} \\
& \qquad \quad \left. + \left\lVert \frac{1}{1 + \pi} \right\rVert_{\L^\infty (\BR_+; \L^\infty (\BR^n))} \lVert \pi \rVert_{\L^\infty (\BR_+; \L^\infty (\BR^n))} \lVert \bmm \rVert_{\L^p (\BR_+; \H^{s + 1, q} (\BR^n))} \right)\\
& \qquad \le C C_0^2 \lVert \bU_0 \rVert_{\H^{s + 1, q} (\BR^n) \times \H^{s, q} (\BR^n)}^2 \\
& \bigg\lVert \int_0^t T (t - \tau) \kappa \dv (\pi \Delta \pi) \,\mathrm{d}\tau \bigg\rVert_{\L^p (\BR_+; \H^{s + 1, q} (\BR^n))} \\
& \qquad \le C \left\lVert \pi \Delta \pi \right\rVert_{\L^p (\BR_+; \H^{s, q} (\BR^n))} \\
& \qquad = C \bigg\lVert \sum_{j = 1}^n \frac{1}{2} \pd_j (\pd_j \pi^2) \bigg\rVert_{\L^p (\BR_+; \H^{s, q} (\BR^n))} \\
& \qquad \le C \left\lVert \pi^2 \right\rVert_{\L^p (\BR_+; \H^{s + 2, q} (\BR^n))} \\
& \qquad \le C \lVert \pi \rVert_{\L^p (\BR_+; \H^{s + 2, q} (\BR^n))} \lVert \pi \rVert_{\L^\infty (\BR_+; \L^\infty (\BR^n))} \\
& \qquad \le C C_0^2 \lVert \bU_0 \rVert_{\H^{s + 1, q} (\BR^n) \times \H^{s, q} (\BR^n)}^2 \\
& \bigg\lVert \int_0^t T (t - \tau) \kappa \dv \left\{\left(- \frac{\lvert \nabla \pi \rvert^2}{2}\right) \bI - \nabla \pi \otimes \nabla \pi \right\} \bigg\rVert_{\L^p (\BR_+; \H^{s + 1, q} (\BR^n))} \\
& \qquad \le C \lVert (\nabla \pi)^2 \rVert_{\L^p (\BR_+; \H^{s, q} (\BR^n))} \\
& \qquad \le C \lVert \nabla \pi \rVert_{\L^p (\BR_+; \H^{s, q} (\BR^n))} \lVert \nabla \pi \rVert_{\L^\infty (\BR_+; \L^\infty (\BR^n))} \\
& \qquad \le C C_0^2 \lVert \bU_0 \rVert_{\H^{s + 1, q} (\BR^n) \times \H^{s, q} (\BR^n)}^2.
\end{align*}}\noindent		
Combining the estimates above, we obtain \eqref{256}. Similarly, we have \eqref{257}, which implies \eqref{254}. \par
We now turn to prove \eqref{255}. To this end, it is enough to show the estimates
\begin{equation}
\label{259}
\begin{split}
\bigg\lVert \int_0^t T(t - \tau) (\bF_1 (\tau) - \bF_2 (\tau))\,\mathrm{d}\tau \bigg\rVert_{\L^p (\BR_+; \H^{s + 1, q} (\BR^n))} & \le C_2 \lVert \bU_1 - \bU_2 \rVert_{X_{p, q, s}}, \\
\bigg\lVert \int_0^t T(t - \tau) (\bF_1 (\tau) - \bF_2 (\tau))\,\mathrm{d}\tau \bigg\rVert_{\W^{1, p} (\BR_+; \H^{s - 1, q} (\BR^n))} & \le C_2 \lVert \bU_1 - \bU_2 \rVert_{X_{p, q, s}}
\end{split}
\end{equation}
hold for $\bU_1, \bU_2 \in X_{p, q, s}$, where we have set $\bF_1 (\tau) := \bF (\pi_1 (x, \tau), \bmm_1 (x, \tau))$ and $\bF_2 (\tau) := \bF (\pi_2 (x, \tau), \bmm_2 (x, \tau))$. Since we have $\pi_1, \pi_2 \in \H^{s, q} (\BR^n)$, see \eqref{258}, Lemmas~\ref{lem-2.5.2}, \ref{lem-2.5.4} and the condition \eqref{253} imply
{\allowdisplaybreaks
\begin{align*}
& \lVert G (\pi_1) - G (\pi_2) \rVert_{\H^{s, q} (\BR^n)} \\
&\quad \le C \lVert \pi_1 - \pi_2 \rVert_{\H^{s, q} (\BR^n)} \left(\lVert \pi_1 \rVert_{\H^{s, q} (\BR^n)} + \lVert \pi_2 \rVert_{\H^{s, q} (\BR^n)} \right) \\
& \quad \quad + C \lVert \pi_1 - \pi_2 \rVert_{\H^{s, q} (\BR^n)} \left(\lVert \pi_1 \rVert_{\H^{s, q} (\BR^n)} + \lVert \pi_2 \rVert_{\H^{s, q} (\BR^n)} \right)^s \\
& \quad \le C \lVert \pi_1 - \pi_2 \rVert_{\H^{s, q} (\BR^n)} \left(\lVert \pi_1 \rVert_{\L^\infty (\BR_+; \B^{s + 2 - 2/p}_{q, p} (\BR^n))} + \lVert \pi_2 \rVert_{\L^\infty (\BR_+; \B^{s + 2 - 2/p}_{q, p} (\BR^n))} \right) \\
& \quad \quad + C \lVert \pi_1 - \pi_2 \rVert_{\H^{s, q} (\BR^n)} \left(\lVert \pi_1 \rVert_{\L^\infty (\BR_+; \B^{s + 2 - 2/p}_{q, p} (\BR^n))} + \lVert \pi_2 \rVert_{\L^\infty (\BR_+; \B^{s + 2 - 2/p}_{q, p} (\BR^n))} \right)^s \\
& \quad \le C \lVert \pi_1 - \pi_2 \rVert_{\H^{s, q} (\BR^n)}, \\
& \lVert G (\pi_1) - G (\pi_2) \rVert_{\B^{s + 1 - 2/p}_{q, p} (\BR^n)} \\
& \quad \le C \lVert \pi_1 - \pi_2 \rVert_{\B^{s + 1 - 2/p}_{q, p} (\BR^n)} \left(\lVert \pi_1 \rVert_{\B^{s + 1 - 2/p}_{q, p} (\BR^n)} + \lVert \pi_2 \rVert_{\B^{s + 1 - 2/p}_{q, p} (\BR^n)} \right) \\
& \quad \le C \lVert \pi_1 - \pi_2 \rVert_{\B^{s + 1 - 2/p}_{q, p} (\BR^n)} \\
& \quad \quad \cdot \left(\lVert \pi_1 \rVert_{\L^\infty (\BR_+; \B^{s + 2 - 2/p}_{q, p} (\BR^n))} + \lVert \pi_2 \rVert_{\L^\infty (\BR_+; \B^{s + 2 - 2/p}_{q, p} (\BR^n))} \right) \\
& \quad \le C \lVert \pi_1 - \pi_2 \rVert_{\B^{s + 1 - 2/p}_{q, p} (\BR^n)}
\end{align*}	
}\noindent
because we have the embedding $\B^{s + 2 - 2/p}_{q, p} (\BR^n) \hookrightarrow \B^{s + 1 - 2/p}_{q, p} (\BR^n) \hookrightarrow \H^{s, q} (\BR^n) \hookrightarrow \L^\infty (\BR^n)$ and the estimate  similar to \eqref{est-pi}. Thus, we observe that
{\allowdisplaybreaks \begin{align*}
& \bigg\lVert \int_0^t T (t - \tau) \Big(G (\pi_1) \nabla \pi_1 - G (\pi_2) \nabla \pi_2 \Big) \,\mathrm{d}\tau \bigg\rVert_{\L^p (\BR_+; \H^{s + 1, q} (\BR^n))} \\
& \le C \lVert (G (\pi_1) - G (\pi_2)) \nabla \pi_1 + G (\pi_2) \nabla (\pi_1 - \pi_2) \rVert_{\L^p (\BR_+; \H^{s - 1, q} (\BR^n))} \\
& \le C \Big(\lVert G (\pi_1) - G (\pi_2) \rVert_{\L^p (\BR_+; \H^{s, q} (\BR^n))} \lVert \nabla \pi_1 \rVert_{\L^\infty (\BR_+; \L^\infty (\BR^n))} \\
& \quad + \lVert G (\pi_1) - G (\pi_2) \rVert_{\L^\infty (\BR_+; \L^\infty (\BR^n))} \lVert \nabla \pi_1 \rVert_{\L^p (\BR_+; \H^{s - 1, q} (\BR^n))} \\
& \quad + \lVert G (\pi_2) \rVert_{\L^p (\BR_+; \H^{s - 1, q} (\BR^n))} \lVert \nabla (\pi_1 - \pi_2) \rVert_{\L^\infty (\BR_+; \L^\infty (\BR^n))} \\
& \quad + \lVert G (\pi_2) \rVert_{\L^\infty (\BR_+; \L^\infty (\BR^n))} \lVert \nabla (\pi_1 - \pi_2) \rVert_{\L^p (\BR_+; \H^{s - 1, q} (\BR^n))} \Big) \\
& \le C \Big(\lVert \pi_1 - \pi_2 \rVert_{\L^p (\BR_+; \H^{s, q} (\BR^n))} \lVert \nabla \pi_1 \rVert_{\L^\infty (\BR_+; \B^{s + 1 - 2/p}_{q, p} (\BR^n))} \\
& \quad + \lVert \pi_1 - \pi_2 \rVert_{\L^\infty (\BR_+; \B^{s + 1 - 2/p}_{q, p} (\BR^n))} \lVert \nabla \pi_1 \rVert_{\L^p (\BR_+; \H^{s - 1, q} (\BR^n))} \\
& \quad + \lVert \pi_2 \rVert_{\L^p (\BR_+; \H^{s - 1, q} (\BR^n))} \lVert \nabla (\pi_1 - \pi_2) \rVert_{\L^\infty (\BR_+; \B^{s + 1 - 2/p}_{q, p} (\BR^n))} \\
& \quad + \lVert \pi_2 \rVert_{\L^\infty (\BR_+; \B^{s + 2 - 2/p}_{q, p} (\BR^n))} \lVert \nabla (\pi_1 - \pi_2) \rVert_{\L^p (\BR_+; \H^{s - 1, q} (\BR^n))} \Big) \\
& \le C \lVert \bU_0 \rVert_{\H^{s + 1, q} (\BR^n) \times \H^{s, q} (\BR^n)} \lVert \bU_1 - \bU_2 \rVert_{X_{p, q, s}}.
\end{align*}}\noindent		
The rest terms can be shown by employing the similar arguments above, which yields \eqref{259}.
\end{proof}

\begin{proof}[Proof of Theorem \ref{Th-2.1.1}]
Define the map $\Phi$ by
\begin{align*}
\Phi (\bU) (t) = T(t) \bU_0 - \int_0^t T(t - \tau) \bN (\tau) \,\mathrm{d}\tau.
\end{align*}
According to the estimate \eqref{252} and Lemma \ref{lem-2.5.1}, for all $\bU$, we see that
\begin{align*}
\lVert \Phi (\bU) \rVert_{X_{p, q, s}} & \le C_0 \lVert \bU_0 \rVert_{\H^{s + 1, q} (\BR^n) \times \H^{s, q} (\BR^n)} + C_1 \lVert \bU_0 \rVert_{\H^{s + 1, q} (\BR^n) \times \H^{s, q} (\BR^n)}^2 \\
& \le (C_0 + C_1 \lVert \bU_0 \rVert_{\H^{s + 1, q} (\BR^n) \times \H^{s, q} (\BR^n)}) \lVert \bU_0 \rVert_{\H^{s + 1, q} (\BR^n) \times \H^{s, q} (\BR^n)}
\end{align*}
On the other hand, we observe that
\begin{align*}
\lVert \Phi (\bU_1) - \Phi (\bU_2) \rVert_{X_{p, q, s}} & \le \bigg\lVert \int_0^t T(t - \tau) (\bN_1 (\tau) - \bN_2 (\tau))\,\mathrm{d}\tau \bigg\rVert_{X_{p, q, s}} \\
& \le C_2 C_0 \lVert \bU_0 \rVert_{\H^{s + 1, q} (\BR^n) \times \H^{s, q} (\BR^n)} \lVert \bU_1 - \bU_2 \rVert_{X_{p, q, s}}
\end{align*}
for all $\bU_1, \bU_2 \in X_{p, q, s}$. Taking the initial data $\bU_0 \in \H^{s + 1, q} (\BR^n) \times \H^{s, q} (\BR^n)^n$ such that
\begin{align*}
\lVert \bU_0 \rVert_{\H^{s + 1, q} (\BR^n) \times \H^{s, q} (\BR^n)} \le \min \bigg(\frac{1}{2 C_0}, \frac{1}{C_0 C_1}, \frac{1}{2 C_0 C_2}\bigg)
\end{align*}
we deduce that
\begin{align*}
\lVert \Phi (\bU) \rVert_{X_{p, q, s}} & \le 2 C_0 \lVert \bU_0 \rVert_{\H^{s + 1, q} (\BR^n) \times \H^{s, q} (\BR^n)}, \\
\lVert \Phi (\bU_1) - \Phi (\bU_2) \rVert_{X_{p, q, s}} & \le \frac{1}{2} \lVert \bU_1 - \bU_2 \rVert_{X_{p, q, s}}
\end{align*}
assuming that $\bU, \bU_1, \bU_2 \in X_{p, q, s}$. Namely, $\Phi$ is the contraction mapping on $X_{p,q, s}$, so that there exists a unique solution $\bU \in X_{p, q, s}$ satisfying \eqref{2.1.1} with $\varrho = 1 + \pi$. This concludes the proof.
\end{proof}

\appendix
\section{Derivation of the Green matrix}
\noindent
We here give a derivation of the Green matrix. According to \eqref{eq-223}, we have the initial-value problem for $\wh \phi (\xi, \cdot)$:
\begin{align}
\label{231}
\left\{\begin{aligned}
\pd_t^2 \wh \phi + (\mu + \nu) \lvert \xi \rvert^2 \pd_t \wh \phi + (\gamma + \kappa \lvert \xi \rvert^2) \lvert \xi \rvert^2 \wh \phi & = 0, \\
(\pd_t \wh \phi (\xi, 0), \wh \phi (\xi, 0)) & = (- \mathrm{i} \xi \cdot \wh \bu_0 (\xi), \wh \phi_0 (\xi)).
\end{aligned}\right.
\end{align}
We then see that the characteristic equation corresponding to \eqref{231} is given by~\eqref{224}. In the following, let $\lambda_\pm (\xi)$ be the roots of \eqref{224}. \par
We first consider the case $\lambda_+ (\xi) \ne \lambda_- (\xi)$, that is, Case 1 with $\lvert \xi \rvert \ne B$ and Case 2--5. In these cases, $\wh \phi$ is denoted by
\begin{equation}
\label{232}
\begin{split}
\wh \phi & = - \frac{\lambda_- \wh \phi_0 + \mathrm{i} \xi \cdot \wh \bu_0}{\lambda_+ - \lambda_-} e^{\lambda_+ t} + \frac{\lambda_+ \wh \phi_0 + \mathrm{i} \xi \cdot \wh \bu_0}{\lambda_+ - \lambda_-} e^{\lambda_- t} \\
& = \left(\cfrac{\lambda_+ e^{\lambda_- t} - \lambda_- e^{\lambda_+ t}} {\lambda_+ - \lambda_-}\right) \wh \phi_0 - \left(\cfrac{e^{\lambda_+ t} - e^{\lambda_- t}}{\lambda_+ - \lambda_-}\right) (\mathrm{i} \xi \cdot \wh \bu_0).
\end{split}
\end{equation}
On the other hand, by \eqref{eq-223}, we also obtain
\begin{align*}
\pd_t \wh \bu = - (\gamma + \kappa \lvert \xi \rvert^2) \mathrm{i} \xi \wh \phi + (- \mu \lvert \xi \rvert^2 \wh \bu - \nu \xi (\xi \cdot \wh \bu)).
\end{align*}
We decompose $\wh \bu$ into components parallel to and orthogonal to $\xi$:
\begin{align*}
\wh \bu (\xi, t) = a (\xi, t) \frac{\xi}{\lvert \xi \rvert} + \bb (\xi, t),
\end{align*}
where $a := (\xi \cdot \wh \bu) / \lvert \xi \rvert$ is the scalar and $\bb$ is the vector perpendicular to $\xi$. We see that $a$ and $\bb$ enjoy
\begin{align*}
\pd_t a = - (\mu + \nu) \lvert \xi \rvert^2 a - \mathrm{i} (\gamma + \kappa \lvert \xi \rvert^2) \lvert \xi \rvert \wh \phi, \qquad \pd_t \bb = - \mu \lvert \xi \rvert^2 \bb,
\end{align*}
respectively. Hence, we obtain
\begin{equation}
\label{233}
\begin{split}
a (\xi, t) & = e^{- (\mu + \nu) \lvert \xi \rvert^2 t} a (\xi, 0) \\
& \quad + \mathrm{i} e^{- (\mu + \nu) \lvert \xi \rvert^2 t} (\gamma + \kappa \lvert \xi \rvert^2) \lvert \xi \rvert \int_0^t e^{(\mu + \nu) \lvert \xi \rvert^2 \tau} \wh \phi (\xi, \tau) \,\mathrm{d}\tau, \\
\bb (\xi, t) & = e^{- \mu \lvert \xi \rvert^2 t} \bigg(\bI - \frac{\xi {}^\top\! \xi}{\lvert \xi \rvert^2} \bigg) \wh \bu_0.
\end{split}
\end{equation}
Recalling \eqref{232} and noting $\lambda_\pm + (\mu + \nu) \lvert \xi \rvert^2 = - \lambda_\mp$ and $\lambda_+ \lambda_- = (\gamma + \kappa \lvert \xi \rvert^2) \lvert \xi \rvert^2$, we arrive at \eqref{227}. \par
We next focus on Case 6: $\lambda_+ (\xi) \equiv \lambda_- (\xi)$ for any $\xi \in \BR^n \setminus \{0\}$. Since $\lambda_+ = \lambda_- = - A \lvert \xi \rvert^2$ is the double root of \eqref{224}, $\wh \phi$ can be written by 
\begin{align*}
\wh \phi = \alpha e^{- A \lvert \xi \rvert^2 t} + \beta t e^{- A \lvert \xi \rvert^2 t}
\end{align*}
with some scalar functions $\alpha = \alpha(\xi)$ and $\beta = \beta(\xi)$. The initial conditions imply
\begin{align*}
\alpha = \wh \phi_0, \quad \beta = A \lvert \xi \rvert^2 \wh \phi_0 - \mathrm{i} \xi \cdot \wh \bu_0,
\end{align*}
so that we obtain
\begin{align}
\label{234}
\wh \phi = (1 + A \lvert \xi \rvert^2 t) e^{- A \lvert \xi \rvert^2} \wh \phi_0 - t e^{- A \lvert \xi \rvert^2 t} (\mathrm{i} \xi \cdot \wh \bu_0).
\end{align}
Noting $\lambda_\pm + (\mu + \nu) \lvert \xi \rvert^2 = - \lambda_\mp$ and $\lambda_+ \lambda_- = A^2 \lvert \xi \rvert^2$ and using \eqref{233} and \eqref{234}, we have \eqref{2212}.

\section*{Acknowledgments}
\noindent
This work is partially supported by JSPS Grant-in-Aid for JSPS Fellows 19J10168 and Top Global University Project of Waseda University.

%---------------------------------------------------------------------------------------------------------------
%	reference
%---------------------------------------------------------------------------------------------------------------

\end{document}